\documentclass[12pt,a4paper,reqno]{amsart}
\usepackage[utf8]{inputenc}
\usepackage[T1]{fontenc}
\usepackage{indentfirst}
\usepackage[english]{babel}
\usepackage{enumitem}
\usepackage{csquotes}
\usepackage[style=alphabetic,sorting=nyt,maxcitenames=4,maxbibnames=5]{biblatex}
\usepackage{amsmath,amsfonts,amssymb,amsopn,amscd,amsthm}
\usepackage{mathrsfs,bbm}
\usepackage{graphicx}
\usepackage{mathpazo}
\usepackage[dvipsnames]{xcolor}
\usepackage[colorlinks=true,citecolor=DarkOrchid,linkcolor=NavyBlue]{hyperref}
\usepackage{pgfplots}
\usepackage{tikz}
\pgfplotsset{compat=1.17}
\usepackage[a4paper,twoside,margin=0.8in]{geometry}
	\newtheorem{main}{Theorem}
    
    \newtheorem*{scheme}{General Scheme}
    \newtheorem{definition}{Definition}
    \newtheorem{lemma}[definition]{Lemma}
    
    \newtheorem{theorem}[definition]{Theorem}
    \newtheorem{proposition}[definition]{Proposition}
    \newtheorem{corollary}[definition]{Corollary}
    
    \theoremstyle{remark}
    \newtheorem{example}[definition]{Example}
    \newtheorem{remark}[definition]{Remark}

    \makeatletter
    \def\thm@space@setup{\thm@preskip=0.5cm   \thm@postskip=0.5cm}
    \makeatother

\newcommand{\N}{\mathbb{N}}
\newcommand{\Z}{\mathbb{Z}}
\newcommand{\Q}{\mathbb{Q}}
\newcommand{\R}{\mathbb{R}}
\newcommand{\C}{\mathbb{C}}
\newcommand{\prim}{\mathbb{P}}
\newcommand{\Tor}{\mathbb{T}}
\newcommand{\E}{\mathrm{e}}
\newcommand{\I}{\mathrm{i}}
\newcommand{\esper}{\mathbb{E}}
\newcommand{\var}{\mathrm{var}}

\newcommand{\proba}{\mathbb{P}}

\newcommand{\Unit}{\mathrm{U}}
\newcommand{\eps}{\varepsilon}

\newcommand{\DD}[1]{\,d\hspace{-0.3mm}{#1}}

\newcommand{\comment}[1]{}
\renewcommand{\Re}{\mathrm{Re}}

\newcommand{\dkol}{d_{\mathrm{Kol}}}

\setlist[enumerate]{itemsep=10pt,topsep=10pt}
\setlist[itemize]{itemsep=5pt,topsep=5pt}

\title[On the precise deviations of the characteristic polynomial]{On the precise deviations of the characteristic\\ polynomial of a random matrix}
\author{Pierre-Loïc Méliot and Ashkan Nikeghbali}
\date{\today}

\begin{document}

\begin{abstract}
In this paper, using techniques developed in our earlier works on the theory of mod-Gaussian convergence, we  prove precise moderate and large deviation results for the logarithm of the characteristic polynomial of a random unitary matrix. In the case where the unitary matrix is chosen according to the Haar measure, the logarithms of the probabilities of fluctuations of order $A=O(N)$ of the logarithm of the characteristic polynomial have been estimated by Hughes, Keating and O'Connell in \cite{HKO01}. In this work we give an equivalent of the probabilities themselves (without the logarithms), and we do so for the more general case of a matrix from the circular $\beta$ ensemble for any parameter $\beta>0$. In comparison to previous results from \cite{FMN16,DHR19}, we   considerably extend the range of fluctuations for which precise estimates can be written.
\end{abstract}

\maketitle

\tableofcontents

\newpage

\section{Characteristic polynomials of matrices of the circular \texorpdfstring{$\beta$}{beta} ensembles}
The goal of this article is to present precise estimates of moderate and large deviations for the characteristic polynomials of random unitary matrices. In the first Subsection \ref{subsec:random_matrices} of this section, we introduce the relevant models from random matrix theory and we recall the known results regarding their fluctuations. In Subsection \ref{subsec:random_zeta}, we explain the connection between the random matrix models of interest and the Riemann $\zeta$-function. We then present in Subsection \ref{subsec:general_method} a general method in order to prove precise large deviation estimates for a sequence of real random variables $(X_N)_{N \in \N}$ (by precise we mean asymptotic estimates of the probabilities themselves, instead of their logarithms). We conclude our introduction in Subsection \ref{subsec:main_results} by stating our main results, and by giving an outline of the later sections of the paper.
\medskip

\noindent \textbf{Notation.} Throughout the paper, $\beta$ and $\delta$ are positive real numbers, and $N \geq 1$ is a positive integer. It will be convenient to set $$\beta'=\frac{\beta}{2} \qquad;\qquad h=2\delta.$$
The open ball with center $z$ and radius $\eps$ in the complex plane is denoted $\mathcal{B}_{(z,\eps)}$, and the vertical strip of complex numbers $z$ with $a<\Re(z)<b$ is denoted $\mathcal{D}_{(a,b)}$. The whole complex plane is denoted $\C$, and the unit circle $\{z\in \C\,|\,|z|=1\}$ is denoted $\Tor$. Given two sequences $(a_N)_{N \in \N}$ and $(b_N)_{N \in \N}$ of positive real numbers, we write $a_N \ll b_N$ if $\lim_{N \to \infty}\frac{a_N}{b_N} = 0$ (in other words, $a_N=o(b_N)$), and $a_N \lesssim b_N$ if $\limsup_{N \to \infty} \frac{a_N}{b_N}<+\infty$ (in other words, $a_N=O(b_N)$). Several computations and the statements of our main theorems will involve the following smooth functions on $\R_+$:
\begin{align*}
\varphi(s) &= \frac{1}{s}\left(\frac{1}{2}-\frac{1}{s}+\frac{1}{\E^s-1}\right);\\
\phi_\beta(s) &= \varphi(s) - \beta'^2\,\varphi(s\beta');\\
\eta_\beta(s) &= \frac{s\beta'\,\phi_\beta(s)}{(\E^{s\beta'}-1)\,\phi_\beta(0)}.
\end{align*}
We have $\eta_2=0$, and if $\beta \neq 2$, then $\eta_\beta$ is positive, decreasing, integrable and with $\eta_\beta(0)=1$.

\begin{figure}[ht]
\begin{center}        
\begin{tikzpicture}[xscale=1,yscale=3]
\draw [thick, NavyBlue, domain=0.5:9, smooth, samples=500] plot ({\x},{(2+8/(exp(\x)-1)-4/(exp(\x/2)-1))/(exp(\x/2)-1)});
\draw [thick, NavyBlue, domain=0:0.51, smooth, samples=500] plot ({\x},{1-\x/4+\x^3/192});
\draw [NavyBlue] (9.5,0.1) node {$\eta_\beta(x)$};
\draw [->] (-0.2,0) -- (10.5,0);
\draw [->] (0,-0.04) -- (0,1.3);
\draw (5,-0.04) -- (5,0);
\draw (-0.2,1) -- (0,1);
\draw (0,-0.12) node {$0$};
\draw (5,-0.12) node {$5$};
\draw (10,-0.12) node {$10$};
\draw (10,-0.04) -- (10,0);
\draw (-0.5,1) node {$1$};
\draw (-0.4,0) node {$0$};
\end{tikzpicture}
\caption{The function $x \mapsto \eta_\beta(x)$ for a parameter $\beta=1$.}\label{fig:eta_beta}
\end{center}
\end{figure}
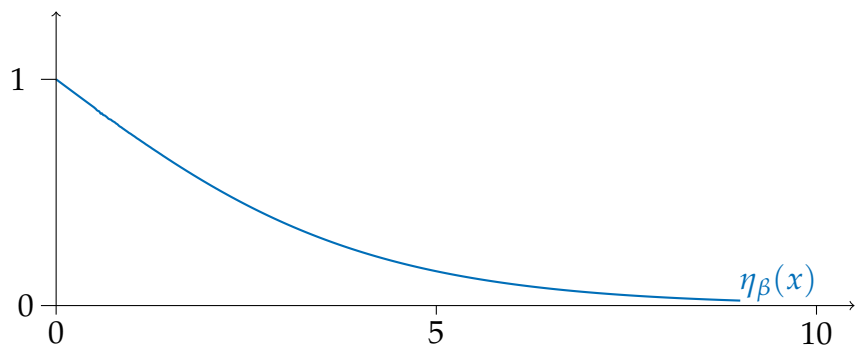

\subsection{The circular \texorpdfstring{$\beta$}{beta} and circular Jacobi \texorpdfstring{$(\beta,\delta)$}{(beta,delta)} ensembles}\label{subsec:random_matrices}
The circular $\beta$ ensemble (in short, C$\beta$E) of order $N$ is the distribution on $N$-tuples $(\E^{\I \theta_1},\ldots,\E^{\I \theta_N})$ of points on the unit circle $\Tor=\R/2\pi\Z$ with density
\begin{equation}
    \frac{1}{C_N(\beta)}\,\prod_{1\leq i<j\leq N}|\E^{\I\theta_i}-\E^{\I \theta_j}|^\beta \DD{\theta_1}\cdots \!\DD{\theta_N},\label{eq:beta_ensemble}
\end{equation}
where the normalisation constant is 
$$C_N(\beta)=(2\pi)^N\,\frac{\Gamma(\beta'N+1)}{\Gamma(\beta'+1)};$$
see for instance \cite[Section 2.8]{For10}. When $\beta=2$, the distribution above is the law of the eigenvalues of a random unitary matrix chosen under the Haar (probability) measure on the unitary group $\Unit(N)$. In the general case of a parameter $\beta>0$, a unitary matrix model with eigenvalue distribution provided by Equation \eqref{eq:beta_ensemble} has been proposed by Killip and Nenciu, see \cite{KN04}. Note that for any parameter $\beta>0$, the distribution of $(\E^{\I\theta_1},\ldots,\E^{\I \theta_N})$ is invariant by multiplication by a phase vector $(z,z,\ldots,z)$ with $z=\E^{\I \theta} \in \Tor$. Therefore, the distribution of the characteristic polynomial
$$P_N(z) = \det(zI_N-U_N) =\prod_{i=1}^N (z-\E^{\I \theta_i})$$
with $z \in \Tor$ does not depend on $z$, and hence without loss of generality we can choose $z=1$.
In this article, we shall be interested in the probabilities of
$$X_N = \log |P_N(1)|=\Re (\log P_N(1)) = \Re \left(\log \left(\prod_{i=1}^N 1-\E^{\I \theta_i} \right)\right)$$
being very large. If one sees the eigenvalues of a unitary random matrix as a system of particles on the unit circle, then the random variable $X_N$ can be considered as the free energy of the system.
The invariance by rotation of the C$\beta$E implies that $\esper_\beta[X_N]=0$ for any $N \geq 1$. On the other hand, the largest possible value is obtained when $U_N=-I_N$ and all the eigenvalues are equal to $-1$, in which case $X_N=N\,\log 2$. Our goal is to obtain the precise asymptotics of 
$\proba_\beta[X_N \geq x]$
for any $x \in [0,N\,\log 2]$. By \emph{precise} we mean that we want an asymptotic equivalent of these probabilies, and not of their logarithms.
\bigskip

To this purpose, it will be useful to generalise a bit the framework described above, and to introduce the circular Jacobi ensembles with parameters $(\beta>0,\delta\geq 0)$. A positive integer $N$ being fixed, the circular Jacobi $(\beta,\delta)$ ensemble (in short, CJ$(\beta,\delta)$E) of order $N$ is the distribution $\proba_{(\beta,\delta)}$ on $N$-tuples $(\E^{\I\theta_1},\ldots,\E^{\I\theta_N})$ of points of the unit circle with density
\begin{equation}\frac{1}{C_N(\beta,\delta)}\,\prod_{1\leq i<j\leq N}|\E^{\I\theta_i}-\E^{\I \theta_j}|^\beta \,\prod_{i=1}^N |1-\E^{\I\theta_i}|^{2\delta} \DD{\theta_1}\cdots \!\DD{\theta_N};\label{eq:circular_jacobi_ensemble}
 \end{equation} 
see \cite{FW00,BNR09} and \cite[Section 3.12]{For10}. The case $\delta=0$ corresponds to the C$\beta$E. On the other hand, the case $\beta=2$, $\delta>0$ corresponds to the so-called Hua--Pickrell measures, see \cite{Hua63,Pic87,Pic91,BO01,Ner02}. A model of random unitary matrices with eigenvalue distribution as in Equation \eqref{eq:circular_jacobi_ensemble} is given in \cite{BNR09}. In particular, given $(\E^{\I \theta_1},\ldots,\E^{\I \theta_N}) \sim \proba_{(\beta,\delta)}$, the theory of deformed Verblunsky coefficients allows one to rewrite the polynomial $P_N(1) = \prod_{i=1}^N (1-\E^{\I \theta_i})$ as a product $\prod_{k=0}^{N-1}(1-\gamma_k)$ of independent random variables, these random variables $\gamma_k$ following explicit distributions on the unit disc $\mathcal{B}_{(0,1)}$ or on the unit circle $\Tor$. The aforementioned paper actually deals with the more general case where $\delta \in \mathcal{D}_{(-\frac{1}{2},+\infty)}$ and the weight 
$$\prod_{i=1}^{N} |1-\E^{\I \theta_i}|^{2\delta} \,\,\text{ is replaced by }\,\,\prod_{i=1}^{N} (1-\E^{-\I \theta_i})^{\delta}(1-\E^{\I \theta_i})^{\overline{\delta}}.$$
Here, we shall only consider the case where $\delta \in \R_+$. The decomposition $P_N(1) = \prod_{k=0}^{N-1}(1-\gamma_k)$ leads to an explicit formula for the Laplace transform of $X_N = \Re (\log P_N(1))$ under the distribution $\proba_{(\beta,\delta)}$. Hence,
\begin{equation}
\esper_{(\beta,\delta)}[\E^{zX_N}] = \prod_{k=0}^{N-1} \frac{\Gamma(\beta'k+1+\delta)^2\,\Gamma(\beta'k+1+2\delta+z)}{\Gamma(\beta'k+1+2\delta)\,\Gamma(\beta'k+1+\delta+\frac{z}{2})^2}, \label{eq:laplace_transform}
\end{equation}
for any $z$ such that $2\delta+\Re(z)>-1$; see \cite[Proposition 4.2]{BNR09}. The asymptotic analysis of this exact formula in various regimes for $z$ and for $\delta$ will be the main technical challenge of this article.
\bigskip

Consider for the moment the special case $\beta=2$ of the C$\beta$E, that is to say the eigenvalue distribution of a Haar distributed unitary matrix. The Heine identity allows one to rewrite the bivariate Laplace transform of the complex random variable $Z_N=\log P_N(1)=X_N+\I Y_N$ as a Toeplitz determinant of size $N\times N$ and associated to a Fisher--Hartwig symbol. The asymptotics of these determinants allowed Hughes, Keating and O'Connell to prove the central limit theorem
\begin{equation}
\frac{Z_N}{\sqrt{\log N}} \rightharpoonup_{\mathrm{law},\,N \to +\infty} \mathcal{N}_{\C}=\mathcal{N}_{\R}\!\left(0,\frac{1}{2}\right)+\I\,\mathcal{N}_{\R}\!\left(0,\frac{1}{2}\right);\label{eq:CLT_HKO}
\end{equation}
see \cite[Theorem 2.1]{HKO01}; this CLT also appears in \cite[Section 2]{KS00}. In particular, for $x$ fixed positive real number,
$$\proba_{2}\!\left[X_N \geq \sqrt{\log N}\,x\right] = \frac{1}{\sqrt{\pi}} \left(\int_{x}^{+\infty} \E^{-y^2}  \DD{y}\right)(1+o(1)).$$
Notice that alternatively, one can prove the central limit theorem by computing the asymptotics of Formula \eqref{eq:laplace_transform} when $N$ goes to infinity and $z$ is fixed. In the setting where $x=sN$ with $s \in (0,\log 2)$, a large deviation principle follows also from the asymptotic analysis of this exact formula:
\begin{equation}
\log \big(\proba_{2}[X_N \geq Nx] \big)= -N^2 \,\Lambda^*(x)\,(1+o(1)),\label{eq:log_large_deviations}
\end{equation}
where $\Lambda^*(x) = \sup_{s\in\R} (xs-\Lambda(s))$ is the Legendre--Fenchel convex dual of the function
$$\Lambda(s) = \begin{cases}
    \frac{1}{2}\,(1+s)^2 \log (1+s) - \left(1+\frac{s}{2}\right)^2 \log \left(1+\frac{s}{2}\right) - \frac{s^2}{4}\,\log(2s) &\text{if }s\geq 0,\\
    +\infty &\text{if }s<0;
\end{cases}$$
see \cite[Theorem 3.3]{HKO01}. Between these two regimes, the logarithms of the probabilities of moderate deviations involve a Gaussian exponent $x^2$: if $(x_N)_{N \in \N}$ is a sequence of positive real numbers such that $\sqrt{\log N}\ll x_N \ll N$, then
\begin{equation}
    \log \big(\proba_{2}[X_N \geq x_N] \big)= - \frac{(x_N)^2}{\log\left(\frac{N}{x_N}\right)}\,(1+o(1)),\label{eq:log_moderate_deviations}
\end{equation}
see \cite[Theorem 3.5]{HKO01}. The results of the present article will improve on these earlier results, by giving the asymptotics of $\proba_2[X_N \geq x_N]$ on the ranges previously described, for a general parameter $\beta>0$. Partial results in this direction were obtained in \cite{FMN16} and \cite{DHR19}; see Subsection \ref{subsec:general_method}.

\subsection{Connection with the Riemann \texorpdfstring{$\zeta$}{zeta}-function}\label{subsec:random_zeta}
When $\beta=2$, the random variables $X_N$ are also meant to predict the behavior of the Riemann $\zeta$-function on the critical line. Let us survey briefly this connection which is one of the main interest of the circular ensembles presented above. Denote $\prim$ the set of prime numbers and $\zeta$ the Riemann zeta function, which is defined on the domain $\{s \in \C\,|\,\Re(s) > 1\}$ by the convergent series and convergent infinite product:
$$\zeta(s)= \sum_{n=1}^\infty \frac{1}{n^s}= \prod_{p \in \prim} \frac{1}{1-p^{-s}},$$
and which is extended by analytic continuation to $\C \setminus \{1\}$. An important part of probabilistic number theory consists in understanding the statistical behavior of the values of $\zeta(s)$ for $s$ in a large domain of the complex plane, for instance a large vertical range $I_{\sigma,T}=\{s=\sigma+\I t,\,\,t \in [T,2T]\}$. Suppose in particular that $\sigma$ is fixed and that $t=U_T$ is chosen uniformly in $[T,2T]$. Then, as $t$ goes to infinity, the random variables $(p^{-\sigma-\I U_T})_{p \in \prim}$ become asymptotically independent and uniformly distributed on the circles $p^{-\sigma}\,\Tor$. If $\sigma>1$, then this joint convergence in law and the absolute convergence of the series $\log (\zeta(s)) = \sum_{p \in \prim} -\log(1-p^{-s})$ on the vertical line $\sigma+\I \R$ imply the existence of a limiting distribution for $\log (\zeta(\sigma + \I U_T))$: we have the weak convergence
$$ \log (\zeta(\sigma+\I U_T)) \rightharpoonup_{\mathrm{law},\,T \to +\infty} \mu(\sigma),$$
and the limiting distribution $\mu(\sigma)$ is the compactly supported distribution of the convergent random series 
$$\sum_{p \in \prim} -\log(1-p^{-\sigma}\,C_p),$$ where $(C_p)_{p \in \prim}$ is a family of independent uniform variables on the circle $\Tor$. By the Kolmogorov two series criterion, this random series is still almost surely convergent for $\sigma > \frac{1}{2}$, and Bohr, Jessen and Wintner extended the convergence result to this setting. For $\frac{1}{2}<\sigma \leq 1$, the limiting distribution $\mu(\sigma)$ of  $\log (\zeta(\sigma+\I U_T))$ is now supported by the whole complex plane; see \cite{BJ30,BJ32,JW35,BJ48} for more details on the properties of $\mu(\sigma)$.\medskip

Suppose now that $\sigma=\frac{1}{2}$. Then, the random series $\sum_{p \in \prim}-\log(1-p^{-\frac{1}{2}}C_p)$ does not converge anymore, but there is still a limiting distribution for $\log(\zeta(\frac{1}{2}+ \I U_T))$, albeit with a renormalisation of these random variables:
\begin{equation}
\frac{\log(\zeta(\frac{1}{2}+\I U_T))}{\sqrt{\log\log T}} \rightharpoonup_{\mathrm{law},\,T \to +\infty} \mathcal{N}_{\C}.\label{eq:CLT_Selberg}
\end{equation}
This is the Selberg central limit theorem; see the papers of Selberg \cite{Sel46,Sel92}, and \cite{Gho83,BH95,RS15} for a detailed account. The obvious analogy between the convergences in law \eqref{eq:CLT_HKO} and \eqref{eq:CLT_Selberg} is a small part of a strong connection between: 
\begin{itemize}
    \item the asymptotic behavior of the characteristic polynomial of a Haar-distributed random unitary matrix;
    \item the asymptotic behavior of the Riemann $\zeta$-function on its critical line $\sigma =\frac{1}{2}$.
\end{itemize}
In particular, the pair correlations of the zeroes of the characteristic polynomial $P_N(z)$ (so, the eigenvalues of $U_N$), which are asymptotically given by the sine-kernel, have been conjectured by Montgomery \cite{Mont73} to also be the asymptotic correlations of the zeroes of the function $\zeta(\frac{1}{2}+\I t)$ on a large range $t \in [T,2T]$. This conjecture has been extended to the higher correlation functions by Rudnick and Sarnak \cite{RS96}. A striking connection  between the moments of the characteristic polynomial (the case $\beta=2$, $\delta=0$ and $z=2k$ of Formula \eqref{eq:laplace_transform}) and the  moments of  $\zeta(\frac{1}{2}+\I U_T)$ has also been conjectured by Keating and Snaith in \cite{KS00}: for any integer $k \geq 1$,
$$\lim_{T \to +\infty} \left(\frac{\esper[|\zeta(\frac{1}{2}+\I U_T)|^{2k}]}{(\log T)^{k^2}}\right) = \left(\lim_{N \to +\infty} \left(\frac{\esper_2[|P_N(1)|^{2k}]}{N^{k^2}}\right) \right)\,\Theta(k),$$
where 
$$\Theta(k) = \prod_{p \in \prim} \left(\sum_{n=0}^{\infty} \left(\frac{\Gamma(n+k)}{n!\,\Gamma(k)}\right)^{\!2}\,(1-p^{-1})^{k^2} p^{-n}\right).$$
We refer to \cite{CFKRS05} for an extension of this conjecture to $L$-functions. Finally, during the last decade, the extrema of the two random fields $(\log |P_N(z)|)_{z \in \Tor}$ and $(\log |\zeta(\frac{1}{2}+\I t)|)_{t \in [T,2T]}$ have been studied and compared, in the framework of log-correlated fields. Fyodorov, Hiary and Keating conjectured in \cite{FHK12,FK14} the convergence in law
$$\max_{z \in \Tor}\left(\log|P_N(z)|\right) - \left(\log N - \frac{3}{4}\,\log \log N\right) \rightharpoonup_{\mathrm{law},\,N\to +\infty} V$$
for some explicit random variable $V$. The tightness of the left-hand side has been established in \cite{CMN18}, see also \cite{ABB17}. Similarly, it has been conjectured that
$$\max_{t \in [T,2T]}\left(\log\left|\zeta\!\left(\frac{1}{2}+\I t\right)\right|\right) - \left(\log \log T - \frac{3}{4}\,\log \log \log T\right) \rightharpoonup_{\mathrm{law},\,T\to +\infty} W$$
for some random variable $W$; partial results in this direction have been obtained in \cite{Naj18,ABBRS19}.
\medskip

We believe that our results of precise large deviations can be extended to the following arithmetic analogues of the random variables $X_N$:
$$\chi_{N} = \Re\left(\sum_{\substack{p \in \prim\\p\leq N}} -\log(1-p^{-\frac{1}{2}}\,C_p)\right)$$
with $(C_p)_{p \in \prim}$ as above. The variables $\chi_N$ are closely related to the Riemann $\zeta$-function, and for instance they have been used in \cite{RS15} in order to give a relatively short proof of the Selberg central limit theorem. We plan to prove in forthcoming works the analogue of our precise large deviation results for $X_N$ for the truncated random $\zeta$-functions $\chi_N$, thereby strengthening the connection between the asymptotics of the characteristic polynomial of random unitary matrices, and the asymptotics of the Riemann $\zeta$-function on the critical line.

\subsection{A general method in order to obtain precise large deviations}\label{subsec:general_method}
Our goal is to make more precise the estimates \eqref{eq:log_large_deviations} and \eqref{eq:log_moderate_deviations}, by computing the asymptotic behavior of the probabilities of large and moderate deviations instead of their logarithms. These kinds of estimates are called \emph{sharp} or \emph{strong} large deviations in the literature. A first step towards such estimates has been made in the papers \cite{FMN16,DHR19} by using the theory of mod-Gaussian convergent sequences. Given a sequence of real-valued random variables $(X_N)_{N \in \N}$ with 
 Laplace transforms well defined on a vertical strip $\mathcal{D}_{(a,b)}$, we say that the sequence converges in the mod-Gaussian sense with parameters $(t_N)_{N \in \N}$ and domain of convergence $\mathcal{D}_{(a,b)}$ if $t_N \to +\infty$ and if, locally uniformly on this domain,
$$\esper[\E^{zX_N}]\,\E^{-\frac{t_Nz^2}{2}} \to_{N \to \infty} \Psi(z),$$
$\Psi(z)$ being a holomorphic function with $\Psi(0)=1$. The case where $z=\I \xi$ is restricted to the imaginary line appeared first in \cite{JKN11}, and the definition with a strip of convergence in the complex plane allows one to obtain large or moderate deviation estimates; see \cite[Definition 1.1.1]{FMN16}. The more general situation where the exponent $\frac{z^2}{2}$ of the Gaussian distribution is replaced by the Lévy--Khintchine exponent $\eta(z)$ of an infinitely divisible distribution is detailed in \cite{DKN15,FMN16,FMN19}. Now a fundamental example of mod-Gaussian convergent sequences is provided by the real parts of the logarithms of the characteristic polynomials of Haar distributed unitary random matrices. This case is quite easier than the general case, because the Laplace transform 
from Equation \eqref{eq:laplace_transform} can then be rewritten in terms of the Barnes $G$-function (\emph{cf.} \cite{Bar00} and \cite[Appendix]{Vor87}). We recall that this entire function is given by the convergent infinite product
$$G(z+1) = \E^{\frac{z(\log (2\pi) - 1) - z^2(\gamma+1)}{2}}\,\prod_{k=1}^\infty \left(1+\frac{z}{k}\right)^k \,\E^{-z+\frac{z^2}{2k}},$$
and that it satisfies the functional equation $G(z+1)=\Gamma(z)\,G(z)$. The Stirling expansion of the Barnes function is 
\begin{equation}
G(1+z) = \exp\left(z^2\left(\frac{1}{2}\,\log z - \frac{3}{4}\right) + z\,\frac{\log 2\pi}{2} - \frac{1}{12}\,\log z + \zeta'(-1)+O\!\left(\frac{1}{|z|}\right)\right);\label{eq:stirling_barnes}
\end{equation}
see for instance \cite[Equations (A.6) and (A.11)]{Vor87}. The functional equation of the Barnes function yields, for any $z$ with $\Re(z)>-1$:
$$\esper_{2}[\E^{zX_N}] = \frac{G(1+N)\,G(1+N+z)}{G(1+N+\frac{z}{2})^2}\,\Psi(z),\quad\text{with }\Psi(z)=\frac{G(1+\frac{z}{2})^2}{G(1+z)}.$$
By injecting the asymptotic estimate \eqref{eq:stirling_barnes} in this formula, we obtain
\begin{equation}
\esper_{2}[\E^{zX_N}] = \E^{\frac{\log N}{4}\,z^2}\,\Psi(z) \,\left(1+O\!\left(\frac{1+|z|^3}{N}\right)\right)\label{eq:mod_gaussian_haar}
\end{equation}
for any $z$ fixed in $\mathcal{D}_{(-1,+\infty)}$. We therefore have a mod-Gaussian convergence with parameters $t_N = \frac{\log N}{2}$ and limiting residue $\Psi(z)$. 
This property was first noticed in \cite[Section 3]{KN12}. In \cite[Theorems 4.15 and 5.1]{DHR19} the mod-Gaussian convergence of the logarithms of the moduli of the characteristic polynomials has been generalised by Dal Borgo, Hovhannisyan and Rouault to the case of a random matrix of a general CJ$(\beta,\delta)$E. Thus, $\beta$ and $\delta$ being two fixed positive parameters, we have the mod-Gaussian convergence
$$\esper_{(\beta,\delta)}\!\left[\E^{z(X_N-\frac{2\delta}{\beta}\log N)}\right] = \E^{\frac{\log N}{2\beta}\,z^2}\,\Psi_{(\beta,\delta)}(z)\,(1+o(1))$$
on the domain $\mathcal{D}_{(-\frac{1}{3},\infty)}$ and for some explicit functions $\Psi_{(\beta,\delta)}$ which can be expressed in terms of the Barnes and Gamma functions. 
Now, a general result of moderate or large deviations in the setting of mod-Gaussian convergent sequences is the following: if $(X_N)_{N \in \N}$ is mod-Gaussian convergent on $\mathcal{D}_{(a,b)}$ with parameters $(t_N)_{N \in \N}$ and limiting residue $\Psi(z)$, then, assuming $a<0<b$ and $(t_N)^{-\frac{1}{2}} \ll x < b$, we have
\begin{equation}
\proba[X_N \geq x\,t_N] = \frac{\E^{-\frac{t_N\,x^2}{2}}}{x\sqrt{2\pi t_N}}\,\Psi(x)\,(1+o(1));\label{eq:mod_gaussian_moderate_deviations}
\end{equation}
see \cite[Theorem 4.2.1]{FMN16}. 
As a consequence, looking at the log-characteristic polynomial of a Haar distributed unitary matrix, we see that for any sequence $(x_N)_{N \in \N}$ with $\sqrt{\log N} \ll x_N \lesssim \log N$, 
\begin{equation}
\proba_{2}[X_N \geq x_N] = \frac{\E^{-\frac{(x_N)^2}{\log N}}}{2x_N}\,\sqrt{\frac{\log N}{\pi}}\,\Psi\!\left(\frac{2x_N}{\log N}\right) (1+o(1)), \label{eq:moderate_deviation_small_range_haar}
\end{equation}
see Theorem 7.5.1 in \emph{loc.~cit.} By setting $x_N=\sqrt{\frac{\log N}{2}}\, y_N$, we see that for $\sqrt{\log N} \lesssim x_N \ll \log N$, the Gaussian estimate
$$\proba_{2}[X_N \geq x_N] = \proba\!\left[\mathcal{N}_\R\!\left(0,\frac{\log N}{2}\right) \geq x_N\right]\,(1+o(1))$$
holds; whereas at the scale $x_N = O(\log N)$, a multiplicative factor $\Psi(\frac{2x_N}{\log N})$ measures the difference between the two probabilities. Similarly, for the matrices of the CJ$(\beta,\delta)$E, the mod-Gaussian convergence result of Dal Borgo--Hovhannisyan--Rouault leads to the following estimate of moderate deviations:
\begin{equation}
\proba_{(\beta,\delta)}\left[X_N \geq \frac{2\delta}{\beta}\log N + x_N\right] = \frac{\E^{-\frac{\beta\,(x_N)^2}{2\log N}}}{x_N}\,\sqrt{\frac{\log N}{2\pi\beta }}\,\Psi_{(\beta,\delta)}\!\left(\frac{\beta x_N}{\log N}\right) (1+o(1))\label{eq:moderate_deviation_small_range_jacobi}
\end{equation}
for any sequence $(x_N)_{N \in \N}$ with $\sqrt{\log N} \ll x_N \lesssim \log N$; see \cite[Theorem 4.16]{DHR19}. 
\bigskip

These sharp estimates of moderate deviations, which follow readily from Equation \eqref{eq:mod_gaussian_moderate_deviations}, are still far from what we want to prove: indeed, we are interested in fluctuations of size $x_N$ up to $O(N)$, instead of $O(\log N)$. The solution to this problem relies on the two following important observations (Lemmas \ref{lem:mod_tilting} and \ref{lem:control_zone}).

\begin{lemma}[Mod-Gaussian convergence and exponential tilting of measures]\label{lem:mod_tilting}
Suppose that $(X_N)_{N \in \N}$ converges in the mod-Gaussian sense on a domain $\mathcal{D}_{(a,b)}$, with parameters $(t_N)_{N \in \N}$ and limiting function $\Psi(z)$. Consider a real parameter $h \in (a,b)$ such that $\Psi(h)\neq 0$. We introduce the new sequence of variables $(X_{N,h})_{N \in \N}$ with distributions
$$\proba_{N,h}[\!\DD{x}] = \frac{\E^{hx}}{\esper[\E^{hX_N}]}\,\proba_{N}[\!\DD{x}],$$
$\proba_N$ being the law of $X_N$. The sequence $(X_{N,h}-t_Nh)_{N \in \N}$ converges again in the mod-Gaussian sense with parameters $(t_N)_{N \in \N}$, domain of convergence $\mathcal{D}_{(a-h,\,b-h)}$, and limiting function $\frac{\Psi(z+h)}{\Psi(h)}$.
\end{lemma}

\begin{proof}
This result originally appeared in \cite[Lemma 4.2.5]{FMN16}, in the more general case of mod-$\phi$ convergent sequences. Set $\Psi_N(z) = \esper[\E^{zX_N}]\,\exp(-\frac{t_N\,z^2}{2})$. We have 
$$\esper[\E^{zX_{N,h}}] = \frac{\esper[\E^{(z+h)X_N}]}{\esper[\E^{hX_N}]} = \E^{\frac{t_N((z+h)^2-h^2)}{2}}\,\frac{\Psi_N(z+h)}{\Psi_N(h)}=\E^{z(t_Nh)}\,\E^{\frac{t_Nz^2}{2}}\,\frac{\Psi_N(z+h)}{\Psi_N(h)}.$$
The result follows by local uniform convergence of the residues $\Psi_N$ towards $\Psi$, since $\Psi(h)\neq 0$.
\end{proof}

\begin{example}\label{ex:hua_pickrell}
The case $\beta=2$ and $\delta \geq 0$ (Hua--Pickrell measures) of Equation \eqref{eq:moderate_deviation_small_range_jacobi} follows immediately from the mod-Gaussian convergence of $(X_N)_{N \in \N}$ under $\proba_{2}$ (Equation \eqref{eq:mod_gaussian_haar}), and from the lemma above. Indeed,  the sequence $(X_N)_{N \in \N}$ under the Hua--Pickrell distribution with parameter $\delta$ is obtained from the same sequence under the Haar measure by an exponential change of measure of parameter $h=2\delta$. Consequently, if $\Psi(z)=\Psi_{(2,0)}(z)$ is the residue previously computed in the Haar case, then
$
\Psi_{(2,\delta)}(z) = \frac{\Psi(z+2\delta)}{\Psi(2\delta)}$.
\end{example}
\medskip

Let us note that if the tilting parameter $h=h_N$ goes to infinity in such a way that the ratio $\Psi_N(z+h_N)/\Psi_N(h_N)$ admits a non-trivial limit, then we still have a mod-Gaussian convergence. This observation opens the way for an extension of the range of parameters $x_N$ for which an estimate of large deviations such as \eqref{eq:moderate_deviation_small_range_haar} or \eqref{eq:moderate_deviation_small_range_jacobi} can be proved. If we also allow the variance parameter $t_N$ to be modified when estimating the Laplace transform of the tilted random variable $X_{N,h_N}$, then the range for the parameters $x_N$ can be even larger. The final nail on the coffin of the restrictions for $x_N$ is the following second observation: in order to get the sharp estimate \eqref{eq:mod_gaussian_moderate_deviations}, during the proof of \cite[Theorem 4.2.1]{FMN16}, we only used an upper bound on the Kolmogorov distance
$$\dkol\left(\frac{X_{N,h_N} - \esper[X_{N,h_N}]}{\sqrt{ \var(X_{N,h_N}) }},\,\mathcal{N}_\R(0,1)\right)$$
 stemming from the mod-Gaussian convergence of the tilted sequence $(X_{N,h_N})_{N \in \N}$. However, such estimates hold \emph{even if strictly speaking we do not have mod-Gaussian convergence}.

\begin{lemma}[Berry--Esseen estimates from a zone of control]\label{lem:control_zone}
Let $(X_N)_{N \in \N}$ be a sequence of random variables such that, for any $\xi \in \R$, we have
$$\esper\!\left[\E^{\I \xi (X_N-\esper[X_N])}\right] = \E^{-\frac{t_N\,\xi^2}{2}}\,\exp(O(|\xi|^3)),$$
with an implied constant $M$ for the $O(\cdot)$. Set $V_N=\frac{X_N-\esper[X_N]}{\sqrt{\var(X_N)}}$. There exists $C>0$ such that
$$\dkol(V_N,\mathcal{N}_{\R}(0,1))= \sup_{s \in \R} \left|\proba[V_N \leq s]-\frac{1}{\sqrt{2\pi}}\int_{-\infty}^s \E^{-\frac{u^2}{2}}\DD{u}\right|\leq \frac{CM}{(t_N)^{\frac{3}{2}}}.$$
\end{lemma}

\begin{proof}
We are in the situation of \cite[Definition 5]{FMN19}, with the following parameters: 
$$\alpha=2,\,\,\, c^\alpha=\frac{1}{2},\,\,\, v=w=3,\,\,\, K_1=K_2=M,\,\,\, \gamma=1,\,\,\, K=\frac{1}{4M}.$$
Indeed, if $\theta_N(\xi)=\esper[\E^{\I \xi (X_N-\esper[X_N])}] \,\E^{\frac{t_N\,\xi^2}{2}}$, then $\theta_N(\xi)=\exp(u_N(\xi))$ with $|u_N(\xi)|\leq M|\xi|^3$, so
$|\theta_N(\xi)-1| \leq |u_N(\xi)|\,\exp(|u_N(\xi)|)\leq M|\xi|^3\,\exp(M|\xi|^3)$
for any $\xi \in \R$. The notion of zone of control leads one to only use this estimate on the domain $[-Kt_N,Kt_N]$ with $K=\frac{1}{4M}$ (see Condition (Z2) in \emph{loc.~cit.}). Then, Equation (5) in \emph{loc.~cit.} ensures the claimed inequality
for some universal constant $C$ (choosing appropriately the parameter $\lambda$ in the aforementioned equation from \cite{FMN19} gives $C \leq 14$). 
\end{proof}
\medskip

Let us now describe a general scheme in order to prove sharp deviation estimates for a sequence of real random variables $(X_N)_{N \in \N}$. This scheme is inspired by classical arguments used in the proofs of the Cramér theorem and of the Bahadur--Rao estimates of strong large deviations for sums of i.i.d.~random variables; see \cite[Theorems 2.2.3 and 3.7.4]{DZ98} and \cite{BR60,CS93}; and the aforementioned result from \cite{FMN16} is a particular case of the general scheme.

\begin{scheme}
Consider a sequence of centered real random variables $(X_N)_{N \in \N}$. The following steps enable the calculation of an asymptotic equivalent of $\,\proba[X_N \geq a_N]\,$ for parameters $0<a_N < M_N = \|X_N\|_\infty$.
\begin{enumerate}[label=\textbf{Step \arabic*.}]
    \item Given $a_N \in (0,M_N)$, find (an asymptotic expansion of) the tilting parameter $h_N$ such that
    $$\esper[\E^{zX_{N,h_N}}] = \frac{\esper[\E^{(z+h_N)X_N}]}{\esper[\E^{h_NX_N}]}\qquad;\qquad \esper[X_{N,h_N}] = a_N.$$
    Compute (an asymptotic expansion of) the variance $v_N = \var(X_{N,h_N})$.
    \item Use the lemma \ref{lem:control_zone} of zone of control  in order to compute an upper bound on 
    $$\dkol\left(\frac{X_{N,h_N} - a_N}{\sqrt{v_N}}, \,\mathcal{N}_\R(0,1)\right)\lesssim\eps_N.$$
    In particular, identify those parameters $a_N$ for which the distance goes to $0$ (asymptotic normality after tilting), and those parameters for which it goes to $0$ faster than $\frac{1}{h_N\sqrt{v_N}}$ (strong asymptotic normality after tilting).
    \item If we have strong asymptotic normality after tilting with $\eps_N\ll \frac{1}{h_N\sqrt{v_N}} \ll 1$, then we have the asymptotic estimate:
    $$\proba[X_N \geq a_N] =\frac{\esper[\E^{h_N(X_{N}-a_N)}]}{h_N \sqrt{2\pi v_N}}\,(1+o(1)).$$
    If we only have asymptotic normality after tilting with $\eps_N \lesssim \frac{1}{h_N\sqrt{v_N}} \ll 1$, then we have an upper bound which is sharp up to a multiplicative constant:
     $$\proba[X_N \geq a_N] \lesssim \frac{\esper[\E^{h_N(X_{N}-a_N)}]}{h_N \sqrt{2\pi v_N}}.$$
\end{enumerate}
\end{scheme}

\noindent This general scheme requires good estimates of the Laplace transform $\esper[\E^{zh_N}]$ for complex parameters $z$ with real part $\Re(z)=h_N$, where $h_N$ is given by the first step. Notice also that we need $h_N \sqrt{v_N} \to +\infty$ for the third step; usually, this will be the case when $a_N$ is large enough, and the estimates for $a_N$ small are covered by the central limit theorem satisfied by the sequence $(X_N)_{N \in \N}$.

\begin{proof}[Proof of the validity of the scheme]
Denote $\proba_N[\!\DD{x}]$ the distribution of $X_N$, and $\proba_{N,h_N}[\!\DD{x}]$ the distribution of the tilted random variable $X_{N,h_N}$. The exponential change of measure with parameter $h_N$ relates the two distributions by:
$$\proba_{N,h_N}[\!\DD{x}] = \frac{\E^{h_N x}}{\esper[\E^{h_NX_N}]}\,\proba_N[\!\DD{x}].$$
Therefore, 
$$\esper[\E^{z X_{N,h_N}}] = \int_{\R} \E^{zx}\,\proba_{N,h_N}[\!\DD{x}] = \frac{1}{\esper[\E^{h_NX_N}]} \int_{\R} \E^{(z+h_N)X_N}\,\proba_N[\!\DD{x}] =  \frac{\esper[\E^{(z+h_N)X_N}]}{\esper[\E^{h_NX_N}]}$$
assuming that the Laplace transforms are convergent. The parameters $h_N$, $a_N$ and $v_N$ satisfy
\begin{align*}
a_N &= \esper[X_{N,h_N}] = \left.\frac{d(\esper[\E^{zX_{N,h_N}}])}{dz}\right|_{z=0}= \frac{\esper[X_N\,\E^{h_NX_N}]}{\esper[\E^{h_NX_N}]} = \left.\frac{d(\log\esper[\E^{zX_{N}}])}{dz}\right|_{z=h_N}; \\
v_N &= \var(X_{N,h_N}) = \left.\frac{d(a_N)}{dz}\right|_{z=h_N} = \left.\frac{d^2(\log\esper[\E^{zX_{N}}])}{dz^2}\right|_{z=h_N}.
\end{align*}
Notice that, when $h$ goes from $0$ to $+\infty$, the expectation $\esper[X_{N,h}]$ increases from $0$ to $M_N = \|X_N\|_\infty$. Therefore, there is a unique solution to the equation $a_N=\esper[X_{N,h_N}]$ for $0<a_N<M_N$, and the parameters make sense. Let us now prove the asymptotic estimates of the third step, under the hypothesis of asymptotic normality after tilting ($\eps_N = o(1)$). We set
$$d\Q_N[x] = d\proba_{N,h_N}\left[\frac{x-a_N}{\sqrt{v_N}}\right];$$
this is the distribution of the scaled random variable considered in the second step of the general scheme.
We denote:
$$
G_{N}(s) = \int_{-\infty}^s d\mathbb{Q}_{N}[x] \qquad;\qquad G(s) = \frac{1}{\sqrt{2\pi}} \int_{-\infty}^s \E^{-\frac{x^2}{2}}\DD{x}.$$
We have by hypothesis $|G_N(s)-G(s)| = O(\eps_N)$ uniformly in $s \in \R$. We now compute:
\begin{align*}
\proba_{\beta}[X_N \geq a_N] &= \int_{a_N}^\infty \proba_{N}[\!\DD{x}] = \esper[\E^{h_NX_N}]\,\int_{a_N}^\infty \E^{-h_N x}\, \proba_{N,h_N}[\!\DD{x}] \\
&= \esper[\E^{h_NX_N}] \,\E^{-h_N a_N} \int_{0}^\infty \exp(-h_N\sqrt{v_N}\,y)\, \mathbb{Q}_{N}[\!\DD{y}]
\end{align*}
and the integral $I_N$ on the right-hand side is equal to
\begin{align*}
&\int_{0}^\infty \E^{-h_N\sqrt{v_N}\,y}\, \mathbb{Q}_{(\beta,\delta_N)}[\!\DD{y}] = h_N\sqrt{v_N}\,\int_0^\infty\E^{-h_N\sqrt{v_N}\,y}\, (G_{N}(y)-G_{N}(0))\DD{y} \\
&= h_N\sqrt{v_N}\,\int_0^\infty\E^{-h_N\sqrt{v_N}\,y}\, (G(y)-G(0))\DD{y} + O\!\left(\eps_Nh_N\sqrt{v_N} \int_0^\infty\E^{-h_N \sqrt{v_N}\,y}\DD{y}\right)\\
& = \frac{1}{\sqrt{2\pi}} \int_0^\infty \E^{-h_N\sqrt{v_N}\,y-\frac{y^2}{2}}\DD{y} + O(\eps_N) \\
&=\frac{\E^{\frac{(h_N)^2v_N}{2}}}{\sqrt{2\pi}} \int_{h_N\sqrt{v_N}}^\infty \exp\!\left(-\frac{Y^2}{2}\right)\DD{Y} + O(\eps_N)= \frac{1}{h_N\sqrt{2\pi v_N}} + O\left(\frac{1}{(h_N\sqrt{v_N})^3}\right)+  O(\eps_N).
\end{align*}
On the last line, we have used an integration by parts in order to get the classical estimate of the tail of the Gaussian distribution: for $M \to +\infty$,
\begin{align*}
\int_M^\infty \E^{-\frac{Y^2}{2}}\DD{Y} = \frac{\E^{-\frac{M^2}{2}}}{M} - \int_M^\infty \frac{\E^{-\frac{Y^2}{2}}}{Y^2}\DD{Y}\qquad;\qquad \int_M^\infty \E^{-\frac{Y^2}{2}}\DD{Y} = \frac{\E^{-\frac{M^2}{2}}}{M}\left(1+O\!\left(\frac{1}{M^2}\right)\right).
\end{align*}
The result follows immediately in the two cases $\eps_N \ll \frac{1}{h_N\sqrt{v_N}}$ and $\eps_N \lesssim \frac{1}{h_N\sqrt{v_N}}$.
\end{proof}

\subsection{Main results and outline of the paper} \label{subsec:main_results}
Before stating the main results of this paper (Theorems \ref{thm:super_moderate} and \ref{thm:super_large}), let us describe informally the application of the general scheme for computing the probabilities of deviation of $X_N = \log|P_N(1)|$, the real part of the logarithm of the characteristic polynomial of a Haar-distributed unitary matrix $U_N$ with size $N$ ($\beta=2$). In Section \ref{sec:estimations}, we shall prove that the parameters $h_N$, $a_N$ and $v_N$ of the general scheme of approximation are related in this case by the following formul{\ae}:
\begin{align*}
a_N &= \frac{h_N}{2}\,\log\!\left(\frac{N}{2h_N}\right) + N\left(\left(1+\frac{h_N}{N}\right)\log\!\left(1+\frac{h_N}{N}\right) - \left(1+\frac{h_N}{2N}\right)\log\!\left(1+\frac{h_N}{2N}\right)\right)\\
&\quad+\frac{1}{12}\left(\frac{1}{N+\frac{h_N}{2}} - \frac{1}{h_N}-\frac{1}{N+h_N}\right)+O\!\left(\frac{1}{(h_N)^3}\right);\\
v_N &= \frac{1}{2}\log\!\left(\frac{N}{2h_N}\right) + \left(\log\!\left(1+\frac{h_N}{N}\right) - \frac{1}{2}\,\log\!\left(1+\frac{h_N}{2N}\right)\right)\\
&\quad+\frac{1}{12}\left( \frac{1}{(N+h_N)^2} - \frac{1}{2(N+\frac{h_N}{2})^2}+ \frac{1}{(h_N)^2}\right)+O\!\left(\frac{1}{(h_N)^4}\right).
\end{align*}
Suppose in particular that $\log N \ll a_N \ll N$. This assumption will turn out to be equivalent to $1\ll h_N \ll N$, and we shall then be able to simplify the equations above:
\begin{align*}
a_N &=  N\,\theta\!\left(\frac{h_N}{2N}\right) -\frac{1}{12h_N}+o\!\left(\frac{1}{h_N}\right);\\
v_N &= \frac{1}{2}\,\theta'\!\left(\frac{h_N}{2N}\right)+\frac{1}{12(h_N)^2}+o\!\left(\frac{1}{(h_N)^2}\right);\\
h_N&=2N\,\theta^{-1}\left(\frac{a_N}{N}\right) + \frac{1}{12a_N}+ o\!\left(\frac{1}{a_N}\right)
\end{align*}
where $\theta(x)=(1+2x)\,\log(1+2x) - (1+x)\log(1+x)-x\log(4x)$ is a continuous bijection from $\R_+$ to $[0,\log 2)$, and $\theta^{-1}$ is its functional inverse; see Figure \ref{fig:theta}. 

\begin{figure}[ht]
\begin{center}        
\begin{tikzpicture}[xscale=1,yscale=5]
\draw [thick, NavyBlue, domain=0.001:10, smooth, samples=500] plot ({\x},{(1+2*\x)*ln(1+2*\x) - (1+\x)*ln(1+\x) - \x*ln(4*\x)});
\draw [NavyBlue] (10.5,0.7) node {$\theta(x)$};
\draw [->] (-0.2,0) -- (10.5,0);
\draw [->] (0,-0.04) -- (0,1);
\draw (5,-0.04) -- (5,0);
\draw (0,-0.08) node {$0$};
\draw (5,-0.08) node {$5$};
\draw (10,-0.08) node {$10$};
\draw (10,-0.04) -- (10,0);
\draw [dashed] (0,0.693) -- (10,0.693);
\draw (-0.5,0.693) node {$\log 2$};
\draw (-0.4,0) node {$0$};
\end{tikzpicture}
\caption{The function $x \mapsto \theta(x)$.}\label{fig:theta}
\end{center}
\end{figure}
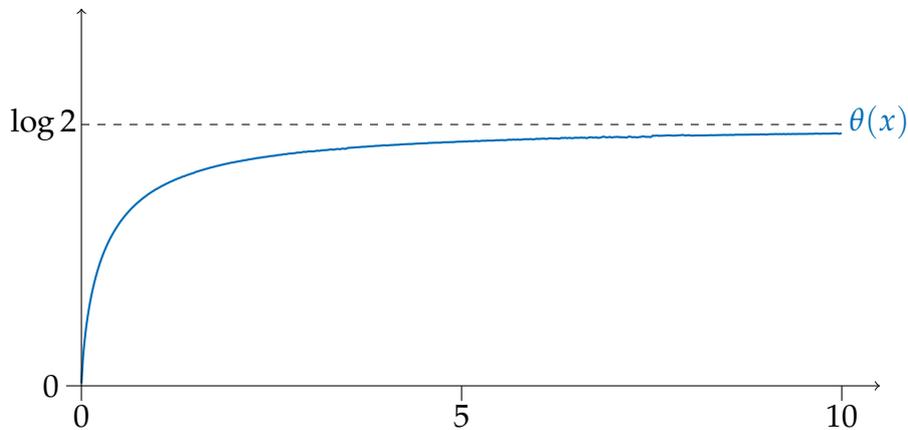

\noindent In a neighborhood of $0$, $\theta(x)$ is equivalent to $x|\log x|$, $\theta^{-1}(x)$ is equivalent to $\frac{x}{|\log x|}$, and $\theta'(x)$ is equivalent to $|\log x|$. Therefore, the equations above imply that for $\log N \ll a_N \ll N$,
$$h_N \simeq_{N \to +\infty} \frac{2a_N }{\log(\frac{N}{a_N})}\qquad;\qquad v_N \simeq_{N \to +\infty} \frac{1}{2}\,\log\!\left(\frac{N}{h_N}\right)\simeq_{N \to +\infty} \frac{1}{2}\,\log\!\left(\frac{N}{a_N}\right).$$
In Section \ref{sec:CLT}, we shall prove the following upper bound on Kolmogorov distances:
$$\dkol\left(\frac{X_{N,h_N}-a_N}{\sqrt{v_N}},\,\mathcal{N}_{\R}(0,1)\right) = O\!\left(\frac{1}{h_N\,(\log(\frac{N}{h_N}))^{\frac{3}{2}}}\right)$$
 for any sequence $(a_N)_{N \in \N}$ such that $\log N \lesssim a_N$. In particular, in the regime $\log N \ll a_N \ll N$, we have
 $$h_N \to +\infty\qquad;\qquad v_N \to +\infty\qquad;\qquad\eps_N = \frac{1}{h_N\,(v_N)^{\frac{3}{2}}} \ll \frac{1}{h_N\,\sqrt{v_N}}\ll 1,$$
 so we have strong asymptotic normality of the tilted sequence $(X_{N,h_N})_{N \in \N}$. The general scheme ensures then that:
$$\proba_2[X_N \geq a_N] = \frac{\esper[\E^{h_N(X_N-a_N)}]}{h_N\sqrt{2\pi v_N}}\,(1+o(1)) = \frac{\esper[\E^{h_N(X_N-a_N)}]}{2a_N}\,\sqrt{\frac{1}{\pi}\,\log\!\left(\frac{N}{a_N}\right)}\,(1+o(1)).$$
Then, it remains to analyse the Laplace transform $\esper[\E^{h_NX_N}]$, and to give an asymptotic equivalent of it. This additional computation will lead to the following result: if $\log N \ll a_N \ll N$ and $\vartheta_N = \theta^{-1}(\frac{a_N}{N})$, then
$$\proba_2[X_N \geq a_N] \simeq_{N \to +\infty} C_2\,(a_N)^{-\frac{13}{12}}\,\left(\log\left(\frac{N}{a_N}\right)\right)^{\!\frac{7}{12}}\,\exp(-f(N,\vartheta_N))$$
 with $C_2 = 2^{-\frac{11}{12}}\,\pi^{-\frac{1}{2}}\,\exp(\zeta'(-1))$, and
 $$f(N,\vartheta_N) = (N\vartheta_N)^2\,\log\!\left(1+\frac{1}{4\vartheta_N(1+\vartheta_N)}\right) + \frac{N^2}{2}\log\!\left(1+\frac{(\vartheta_N)^2}{1+2\vartheta_N}\right).$$
 Moreover, under the stronger hypothesis that $\log N \ll a_N \lesssim N^{\frac{1}{3}}$, the term in the exponential writes as
 $$-f(N,\vartheta_N) = -Na_N\vartheta_N+\frac{(N\vartheta_N)^2}{2}+o(1).$$
We shall also state results  when $\sqrt{\log N} \ll a_N \lesssim \log N$, and when $a_N$ is of order $N$. We shall deal with the case of a general parameter $\beta>0$. In this setting it will sometimes be convenient to modify a tiny bit the function $\theta$, and to set:
$$\theta_{N,\beta}(x) = \theta(x) + \frac{\beta'-1}{2\beta'N} (\log 2 + \log (1+x) - \log(1+2x)).$$
Note that $\theta_{N,2}=\theta$ for any $N$, and also that $\lim_{N \to \infty} \theta_{N,\beta}(x)=\theta(x)$ uniformly on the real line for any fixed parameter $\beta>0$. For $N$ large enough, the function $\theta_{N,\beta}$ is again a continuous bijection, this time from $\R_+$ to $[\frac{\beta'-1}{2\beta'N}\log 2,\log 2)$; see Lemma \ref{lem:modification_theta}.\bigskip

We fix $\beta>0$; in all the estimates of Theorems \ref{thm:super_moderate} and \ref{thm:super_large} below, the $O(\cdot)$'s and $o(\cdot)$'s involve implied constants which are allowed to depend on the parameter $\beta$. In order to make this clear, we add an index $\beta$ to all such estimates, so for instance we shall denote $O_\beta(y)$ a quantity whose module is smaller than $C_\beta|y|$ for some positive constant $C_\beta$ which can only depend on $\beta$.

\begin{main}[Sharp moderate deviations of the characteristic polynomial of the C$\beta$E ensemble]\label{thm:super_moderate}
Let $X_N$ be the real part of the logarithm of the characteristic polynomial of a random matrix from the C$\beta$E, and let $(x_N)_{N \in \N}$ be a sequence of positive numbers such that $\log N \lesssim x_N \ll N$. 

\begin{enumerate}
\item We have:
$$\proba_{\beta}[X_N \geq x_N] = \frac{\E^{-\Lambda_{N,\beta}^*(x_N)}}{x_N}\sqrt{\frac{1}{2\pi\beta}\log\!\left(\frac{N}{x_N}\right)} \left(1+O_\beta\!\left(\frac{1}{\log(\frac{N}{x_N})}\right)\right),$$
where $\Lambda_{N,\beta}^*(\cdot)$ is the Legendre--Fenchel conjugate of $\Lambda_{N,\beta}(h)=\log \esper_\beta[\E^{hX_N}]$, which is a convex function.

\item Suppose that $x_N$ is of order $\log N$. Then, the formula above rewrites as
$$\proba_\beta[X_N \geq x_N] = \frac{\E^{-\frac{\beta\,(x_N)^2}{2\log N}}}{x_N}\sqrt{\frac{\log N}{2\pi\beta}} \,\Psi_\beta\!\left(\frac{\beta x_N}{\log N}\right)\left(1+O_\beta\!\left(\frac{\log \log N}{\log N}\right)\right),$$
with a function $\Psi_\beta$ that can be expressed in terms of the function $\Psi=\Psi_2$:
$$\Psi_\beta(\beta t) = (\Psi(2t))^{\beta'}\,\frac{\Gamma(1+2\beta't)\,\Gamma(1+t)^{\beta'+1}}{\Gamma(1+\beta't)^2\,\Gamma(1+2t)^{\frac{\beta'+1}{2}}}\,\,\exp\!\left(\frac{1-\beta'^2}{12\beta'}\int_0^\infty \frac{(1-\E^{-\beta'ts})^2}{s}\, \eta_\beta(s)\DD{s}\right).$$

\item Suppose now that $\log N \ll x_N \ll N$, and set $\vartheta_N = (\theta_{N,\beta})^{-1}(\frac{x_N}{N})$, which is asymptotically equivalent to $\frac{x_N}{N\log(\frac{N}{x_N})}$. Then, 
\begin{align*}
\proba_\beta[X_N \geq x_N]&= C_\beta\,(x_N)^{\!\frac{\beta'^2-15\beta'+1}{12\beta'}}\left(\log\!\left(\frac{N}{x_N}\right)\right)^{\!\frac{9\beta'-1-\beta'^2}{12\beta'}} \exp\!\left(-f(N,\beta,\vartheta_N)+o(1)\right),
\end{align*}
where $C_\beta$ is an explicit positive constant given by Equation \eqref{eq:cbeta}, and
$$f(N,\beta,\vartheta_N) = \beta' (N \vartheta_N)^2 \log\left(1+\frac{1}{4\vartheta_N(1+\vartheta_N)}\right) + \frac{N^2\beta'-N(\beta'-1)}{2} \log\!\left(1+\frac{(\vartheta_N)^2}{1+2\vartheta_N}\right).$$

\item Suppose more precisely that $\log N \ll x_N \lesssim N^{1/3}$, and set $b_N = N\vartheta_N = N\,(\theta_{N,\beta})^{-1}(\frac{x_N}{N})$, which is asymptotically equivalent to $\frac{x_N}{\log(\frac{N}{x_N})}$. Then, 
$$\proba_\beta[X_N \geq x_N] = C_\beta\,(x_N)^{\!\frac{\beta'^2-15\beta'+1}{12\beta'}}\left(\log\!\left(\frac{N}{x_N}\right)\right)^{\!\frac{9\beta'-1-\beta'^2}{12\beta'}}\,\exp\!\left(\beta'\!\left(- x_Nb_N+\frac{(b_N)^2}{2}\right)+o(1) \right).$$
\end{enumerate} 
\end{main}

\noindent The explicit constant $C_\beta$ above is given by:
\begin{equation}
C_\beta =  \frac{2^{\frac{1}{12\beta'}}\,\pi^{\frac{\beta'-3}{4}}}{\beta}\, \exp\left(\frac{1-\beta'^2}{12\beta'}(A_\beta+\log\beta') + \beta'\zeta'(-1)\right),\label{eq:cbeta}
\end{equation}
with 
$$A_\beta=\int_{s=0}^\infty(1-\E^{-\frac{s}{2}})^2\,\frac{\eta_\beta(s)}{s}\DD{s} + \int_{t=1}^\infty \int_{s=0}^\infty \frac{2\E^{-\frac{st}{2}}-\E^{-st}}{t}\, \eta_\beta'(s)\DD{s}\DD{t}.$$
If $\beta=2$ and $\log N \ll x_N \ll N$, we recover as a particular case of Theorem \ref{thm:super_moderate} the estimates stated informally at the beginning of this paragraph. On the other hand, the asymptotics of $\proba_{\beta}[X_N \geq x_N]$ when $x_N \lesssim \log N$ appear in \cite[Theorem 4.16]{DHR19}, and it is not very difficult to transform the formula of \emph{loc.~cit.} in order to obtain our formula for $\Psi_{\beta}$. When $x_N \gg \log N$, as far as we know, even in the simpler case $\beta=2$, our estimates of $\proba_{\beta}[X_N \geq x_N]$ are new. One of the main interest of these formulas is the appearance of the "non-standard" exponents 
$$\frac{\beta'^2-15\beta'+1}{12\beta'} \quad \text{and}\quad \frac{9\beta'-1-\beta'^2}{12\beta'};$$
here by non-standard we mean different from the usual exponents $\pm1$ or $\pm\frac{1}{2}$ that appear in central limit theorems and in Bahadur--Rao type large deviation estimates. We believe that similar exponents can appear in the more general setting of moderate or large deviations of sequences of random variables that admit a mod-Gaussian renormalisation. In particular, this should be the case for the arithmetic analogues $\chi_N$ of the random variables $X_N$, which are also mod-Gaussian convergent (see \cite{KN12}).\medskip

Let us now consider the regime where $x_N$ is of order $N$. We denote 
$$I(x)=-\frac{(1-4x^2)}{2}\log (1+2x) - x^2 \log (4x) +(1-x^2)\log(1+x),$$
$x$ being an arbitrary positive real number.

\begin{main}[Sharp large deviations of the characteristic polynomial of the C$\beta$E ensemble]\label{thm:super_large} Consider parameters $\alpha_0$ in the interval $[\alpha,\log2]$, where $\alpha$ is an arbitrary fixed postive constant.
\begin{enumerate}
    \item In the same setting as in Theorem \ref{thm:super_moderate}, we have
    $$\limsup_{N \to \infty} \left(\frac{N \,\proba_\beta[X_N \geq \alpha_0 N]}{(\theta^{-1}(\alpha_0))^2}\,\exp(\Lambda_{N,\beta}^*(\alpha_0N)) \right) \leq M_{\alpha,\beta}$$
    for some finite constant $M_{\alpha,\beta}$ which only depends on the two quantities $\alpha$ and $\beta$.
    \item If $\alpha_0$ belongs to the interval $[\alpha,\alpha']$ with $\alpha'$ small enough, then we also have
    $$\liminf_{N \to \infty} \left(N\, \proba_\beta[X_N \geq \alpha_0 N]\,\exp(\Lambda_{N,\beta}^*(\alpha_0N)) \right) \geq \,m_{\alpha,\alpha',\beta}$$
    with another finite constant $m_{\alpha,\alpha',\beta}$ which depends only on $\alpha$, $\alpha'$ and $\beta$ and which is strictly positive.
    \item If $\beta=2$, then
    $$\Lambda_{N,2}^*(\alpha_0N) = N^2 I(\theta^{-1}(\alpha_0)) + \frac{\log N}{12} + O_\alpha(1).$$
    Therefore the upper bound rewrites in this case as:
    $$\limsup_{N \to \infty} \left(\frac{N^{\frac{13}{12}}\,\proba_2[X_N \geq \alpha_0 N]}{(\theta^{-1}(\alpha_0))^2}\,\exp\left(N^2\,I(\theta^{-1}(\alpha_0))\right)\right) \leq M_\alpha<+\infty,$$
    and we have a similar statement for the lower bound if $\alpha_0 \in [\alpha,\alpha']$.
    \item If $\beta\neq 2$, then
    $$\Lambda_{N,\beta}^*(\alpha_0N) = N^2 \beta'\,I(\theta^{-1}(\alpha_0)) + O_{\alpha_0,\beta}(N).$$
    Therefore the sequence of random variables $(\frac{X_N}{N})_{N \in \N}$ satisfies under $\proba_\beta$ a principle of large deviations with speed $N^2$ and rate function $\beta'\,I\circ \theta^{-1}$ (see Figure \ref{fig:rate_function}).
\end{enumerate}  
\end{main}

\begin{figure}[ht]
\begin{center}        
\begin{tikzpicture}[xscale=10,yscale=5]
\draw [->] (-0.02,0) -- (0.75,0);
\draw [->] (0,-0.03) -- (0,1.1);
\draw [dashed] (0.693,-0.03) -- (0.693,1);
\draw (0,-0.07) node {$0$};
\draw (0.693,-0.07) node {$\log 2$};
\draw (-0.02,1) -- (0,1);
\draw (-0.04,0) node {$0$};
\draw (-0.04,1) node {$1$};
\draw (0.67,1.12) [NavyBlue] node {$I\circ\theta^{-1}(x)$};
\draw [thick,NavyBlue] (0.002000, 0) -- (0.004000, 0) -- (0.006000, 0) -- (0.008000, 0) -- (0.01000, 0) -- (0.01200, 0.00002043) -- (0.01400, 0.00002043) -- (0.01600, 0.00004251) -- (0.01800, 0.00004251) -- (0.02000, 0.00004251) -- (0.02200, 0.00007121) -- (0.02400, 0.00007121) -- (0.02600, 0.0001060) -- (0.02800, 0.0001464) -- (0.03000, 0.0001464) -- (0.03200, 0.0001922) -- (0.03400, 0.0001922) -- (0.03600, 0.0002430) -- (0.03800, 0.0002430) -- (0.04000, 0.0002986) -- (0.04200, 0.0003588) -- (0.04400, 0.0003588) -- (0.04600, 0.0004233) -- (0.04800, 0.0004921) -- (0.05000, 0.0004921) -- (0.05200, 0.0005650) -- (0.05400, 0.0006417) -- (0.05600, 0.0006417) -- (0.05800, 0.0007223) -- (0.06000, 0.0008065) -- (0.06200, 0.0008065) -- (0.06400, 0.0008943) -- (0.06600, 0.0009855) -- (0.06800, 0.001080) -- (0.07000, 0.001178) -- (0.07200, 0.001178) -- (0.07400, 0.001279) -- (0.07600, 0.001383) -- (0.07800, 0.001490) -- (0.08000, 0.001600) -- (0.08200, 0.001600) -- (0.08400, 0.001713) -- (0.08600, 0.001829) -- (0.08800, 0.001947) -- (0.09000, 0.002068) -- (0.09200, 0.002192) -- (0.09400, 0.002318) -- (0.09600, 0.002447) -- (0.09800, 0.002578) -- (0.1000, 0.002712) -- (0.1020, 0.002848) -- (0.1040, 0.002986) -- (0.1060, 0.003127) -- (0.1080, 0.003270) -- (0.1100, 0.003415) -- (0.1120, 0.003562) -- (0.1140, 0.003711) -- (0.1160, 0.003862) -- (0.1180, 0.004015) -- (0.1200, 0.004170) -- (0.1220, 0.004327) -- (0.1240, 0.004487) -- (0.1260, 0.004647) -- (0.1280, 0.004810) -- (0.1300, 0.004975) -- (0.1320, 0.005141) -- (0.1340, 0.005479) -- (0.1360, 0.005651) -- (0.1380, 0.005824) -- (0.1400, 0.005999) -- (0.1420, 0.006175) -- (0.1440, 0.006533) -- (0.1460, 0.006714) -- (0.1480, 0.006896) -- (0.1500, 0.007081) -- (0.1520, 0.007453) -- (0.1540, 0.007642) -- (0.1560, 0.007832) -- (0.1580, 0.008216) -- (0.1600, 0.008410) -- (0.1620, 0.008605) -- (0.1640, 0.008999) -- (0.1660, 0.009198) -- (0.1680, 0.009600) -- (0.1700, 0.009803) -- (0.1720, 0.01001) -- (0.1740, 0.01042) -- (0.1760, 0.01063) -- (0.1780, 0.01104) -- (0.1800, 0.01126) -- (0.1820, 0.01168) -- (0.1840, 0.01189) -- (0.1860, 0.01233) -- (0.1880, 0.01254) -- (0.1900, 0.01298) -- (0.1920, 0.01342) -- (0.1940, 0.01364) -- (0.1960, 0.01409) -- (0.1980, 0.01432) -- (0.2000, 0.01477) -- (0.2020, 0.01522) -- (0.2040, 0.01568) -- (0.2060, 0.01591) -- (0.2080, 0.01638) -- (0.2100, 0.01684) -- (0.2120, 0.01731) -- (0.2140, 0.01755) -- (0.2160, 0.01803) -- (0.2180, 0.01850) -- (0.2200, 0.01898) -- (0.2220, 0.01947) -- (0.2240, 0.01995) -- (0.2260, 0.02044) -- (0.2280, 0.02069) -- (0.2300, 0.02118) -- (0.2320, 0.02168) -- (0.2340, 0.02217) -- (0.2360, 0.02267) -- (0.2380, 0.02343) -- (0.2400, 0.02393) -- (0.2420, 0.02444) -- (0.2440, 0.02495) -- (0.2460, 0.02546) -- (0.2480, 0.02598) -- (0.2500, 0.02649) -- (0.2520, 0.02701) -- (0.2540, 0.02779) -- (0.2560, 0.02832) -- (0.2580, 0.02884) -- (0.2600, 0.02963) -- (0.2620, 0.03016) -- (0.2640, 0.03069) -- (0.2660, 0.03149) -- (0.2680, 0.03202) -- (0.2700, 0.03282) -- (0.2720, 0.03336) -- (0.2740, 0.03417) -- (0.2760, 0.03471) -- (0.2780, 0.03553) -- (0.2800, 0.03607) -- (0.2820, 0.03689) -- (0.2840, 0.03744) -- (0.2860, 0.03826) -- (0.2880, 0.03909) -- (0.2900, 0.03992) -- (0.2920, 0.04047) -- (0.2940, 0.04130) -- (0.2960, 0.04213) -- (0.2980, 0.04297) -- (0.3000, 0.04381) -- (0.3020, 0.04465) -- (0.3040, 0.04549) -- (0.3060, 0.04634) -- (0.3080, 0.04718) -- (0.3100, 0.04803) -- (0.3120, 0.04888) -- (0.3140, 0.04973) -- (0.3160, 0.05058) -- (0.3180, 0.05172) -- (0.3200, 0.05257) -- (0.3220, 0.05343) -- (0.3240, 0.05428) -- (0.3260, 0.05543) -- (0.3280, 0.05629) -- (0.3300, 0.05744) -- (0.3320, 0.05830) -- (0.3340, 0.05945) -- (0.3360, 0.06031) -- (0.3380, 0.06146) -- (0.3400, 0.06262) -- (0.3420, 0.06377) -- (0.3440, 0.06464) -- (0.3460, 0.06580) -- (0.3480, 0.06695) -- (0.3500, 0.06811) -- (0.3520, 0.06927) -- (0.3540, 0.07043) -- (0.3560, 0.07159) -- (0.3580, 0.07275) -- (0.3600, 0.07421) -- (0.3620, 0.07537) -- (0.3640, 0.07653) -- (0.3660, 0.07798) -- (0.3680, 0.07915) -- (0.3700, 0.08031) -- (0.3720, 0.08176) -- (0.3740, 0.08321) -- (0.3760, 0.08438) -- (0.3780, 0.08583) -- (0.3800, 0.08728) -- (0.3820, 0.08873) -- (0.3840, 0.09018) -- (0.3860, 0.09163) -- (0.3880, 0.09308) -- (0.3900, 0.09453) -- (0.3920, 0.09597) -- (0.3940, 0.09771) -- (0.3960, 0.09915) -- (0.3980, 0.1006) -- (0.4000, 0.1023) -- (0.4020, 0.1041) -- (0.4040, 0.1055) -- (0.4060, 0.1072) -- (0.4080, 0.1089) -- (0.4100, 0.1107) -- (0.4120, 0.1124) -- (0.4140, 0.1141) -- (0.4160, 0.1161) -- (0.4180, 0.1178) -- (0.4200, 0.1195) -- (0.4220, 0.1215) -- (0.4240, 0.1235) -- (0.4260, 0.1252) -- (0.4280, 0.1272) -- (0.4300, 0.1291) -- (0.4320, 0.1311) -- (0.4340, 0.1331) -- (0.4360, 0.1353) -- (0.4380, 0.1372) -- (0.4400, 0.1395) -- (0.4420, 0.1414) -- (0.4440, 0.1436) -- (0.4460, 0.1459) -- (0.4480, 0.1481) -- (0.4500, 0.1505) -- (0.4520, 0.1527) -- (0.4540, 0.1549) -- (0.4560, 0.1574) -- (0.4580, 0.1598) -- (0.4600, 0.1622) -- (0.4620, 0.1647) -- (0.4640, 0.1671) -- (0.4660, 0.1698) -- (0.4680, 0.1722) -- (0.4700, 0.1748) -- (0.4720, 0.1775) -- (0.4740, 0.1804) -- (0.4760, 0.1830) -- (0.4780, 0.1859) -- (0.4800, 0.1885) -- (0.4820, 0.1913) -- (0.4840, 0.1944) -- (0.4860, 0.1972) -- (0.4880, 0.2003) -- (0.4900, 0.2034) -- (0.4920, 0.2064) -- (0.4940, 0.2094) -- (0.4960, 0.2127) -- (0.4980, 0.2159) -- (0.5000, 0.2192) -- (0.5020, 0.2226) -- (0.5040, 0.2260) -- (0.5060, 0.2295) -- (0.5080, 0.2328) -- (0.5100, 0.2365) -- (0.5120, 0.2401) -- (0.5140, 0.2436) -- (0.5160, 0.2474) -- (0.5180, 0.2512) -- (0.5200, 0.2551) -- (0.5220, 0.2591) -- (0.5240, 0.2630) -- (0.5260, 0.2671) -- (0.5280, 0.2712) -- (0.5300, 0.2755) -- (0.5320, 0.2797) -- (0.5340, 0.2840) -- (0.5360, 0.2886) -- (0.5380, 0.2929) -- (0.5400, 0.2975) -- (0.5420, 0.3022) -- (0.5440, 0.3071) -- (0.5460, 0.3120) -- (0.5480, 0.3168) -- (0.5500, 0.3217) -- (0.5520, 0.3269) -- (0.5540, 0.3322) -- (0.5560, 0.3376) -- (0.5580, 0.3430) -- (0.5600, 0.3485) -- (0.5620, 0.3542) -- (0.5640, 0.3600) -- (0.5660, 0.3659) -- (0.5680, 0.3719) -- (0.5700, 0.3779) -- (0.5720, 0.3842) -- (0.5740, 0.3908) -- (0.5760, 0.3972) -- (0.5780, 0.4038) -- (0.5800, 0.4107) -- (0.5820, 0.4178) -- (0.5840, 0.4249) -- (0.5860, 0.4322) -- (0.5880, 0.4397) -- (0.5900, 0.4474) -- (0.5920, 0.4553) -- (0.5940, 0.4632) -- (0.5960, 0.4715) -- (0.5980, 0.4799) -- (0.6000, 0.4887) -- (0.6020, 0.4976) -- (0.6040, 0.5067) -- (0.6060, 0.5162) -- (0.6080, 0.5259) -- (0.6100, 0.5358) -- (0.6120, 0.5460) -- (0.6140, 0.5566) -- (0.6160, 0.5674) -- (0.6180, 0.5786) -- (0.6200, 0.5901) -- (0.6220, 0.6021) -- (0.6240, 0.6144) -- (0.6260, 0.6270) -- (0.6280, 0.6402) -- (0.6300, 0.6538) -- (0.6320, 0.6680) -- (0.6340, 0.6827) -- (0.6360, 0.6979) -- (0.6380, 0.7137) -- (0.6400, 0.7302) -- (0.6420, 0.7474) -- (0.6440, 0.7655) -- (0.6460, 0.7842) -- (0.6480, 0.8040) -- (0.6500, 0.8246) -- (0.6520, 0.8464) -- (0.6540, 0.8694) -- (0.6560, 0.8936) -- (0.6580, 0.9193) -- (0.6600, 0.9466) -- (0.6620, 0.9758) -- (0.6640, 1.007) -- (0.6660, 1.040);
\end{tikzpicture}
\caption{The rate function $I\circ \theta^{-1}(x)$ of the large deviation principle satisfied by $\frac{X_N}{N}$ when $\beta=2$.\label{fig:rate_function}}
\end{center}
\end{figure}
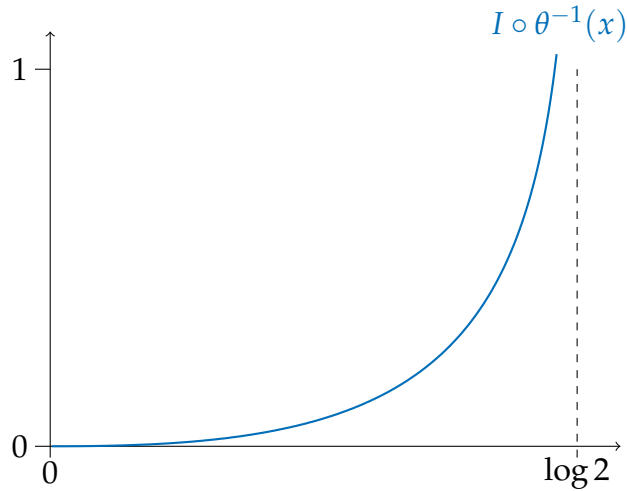

\noindent Let us make a few comments. First, the second part of this strong principle of large deviations only holds if $\alpha_0$ is very small: we are not able to prove that the sequence 
$$ N\, \proba_\beta[X_N \geq \alpha_0 N]\,\exp(\Lambda_{N,\beta}^*(\alpha_0N))$$ 
does not go to $0$ if $\alpha_0$ is very large (for instance, if $\alpha_0$ is close to $\log 2$). Then in the last part of the theorem we could give a more precise asymptotic expansion of the function $\Lambda_{N,\beta}^*(\alpha_0N)$, which would lead to a statement analoguous to the third item. This more precise expression is a bit complicated, and this is why we choose to only state the non-sharp principle of large deviations. This principle of large deviations can also be obtained by using the contraction principle on the large deviation principle which holds for the empirical measures of circular $\beta$ ensembles; see \cite[Theorem 5.5]{BNR09}. Note however that the computation of the rate function is much more difficult when using the contraction principle as one has to minimise under constraints a functional of probability measures. Last when $\beta=2$ the third item of Theorem \ref{thm:super_large} is an improvement of the Hughes--Keating--O'Connell large deviation principle (Equation \eqref{eq:log_large_deviations}): we have an asymptotic expansion of the probability instead of its logarithm, as well as an explicit formula for the rate function (if we consider $\theta^{-1}$ to be explicit).

\begin{remark}
For the sake of simplicity, we choose to state our results of strong moderate and large deviations for the circular $\beta$ ensembles. However, it will be clear from our arguments that the same techniques can be used in order to obtain the strong moderate and large deviations for the circular Jacobi $(\beta,\delta)$ ensembles, with any parameter $\delta>0$.
\end{remark}
\medskip

\noindent \textbf{Outline of the paper.} In Section \ref{sec:estimations}  we give sharp estimates of the Laplace transforms of the random variables $X_N$. These estimates will enable us to control the mean $a_{N,\beta}$ and the variance $v_{N,\beta}$ of these random variables after an exponential change of measure, with a tilting parameter $h_N$ which we shall allow to be very large. An important argument which will be used constantly is a comparison theorem (Theorem \ref{thm:comparison}) which relates the formul{\ae} for a general parameter $\beta$ to those for $\beta=2$. This connection might be of independent interest. In Section \ref{sec:CLT}  we use the notion of zone of control developed in \cite{FMN19} in order to compute the speed of the convergence of the central limit theorem satisfied by the tilted variables $X_{N,h_N}$. We then proceed in Section \ref{sec:sharp_polynomial} to the proofs of our moderate and large deviation principles: in addition to the previous arguments, we prove there some asymptotic expansions for the Legendre--Fenchel convex duals of the log-Laplace transforms of our random variables. This will enable us to make more explicit the asymptotic equivalents of the probabilities of deviations.
\bigskip

\section{Estimation of the mean and of the variance}\label{sec:estimations}
In this section we investigate the relations between the three following quantities:
\begin{itemize}
    \item a sequence of positive parameters $(h_N)_{N \in \N}$, which will be used in order to tilt the random variables $X_N$ under $\proba_\beta$:
    \begin{align*}
    \esper_\beta[\E^{zX_{N,h_N}}] &= \frac{\esper_\beta[\E^{(z+h_N)X_N}]}{\esper_\beta[\E^{h_NX_N}]}.
    \end{align*}
    As already seen in the case $\beta=2$ (Example \ref{ex:hua_pickrell}), the tilted random variable $X_{N,h_N}$ under $\proba_\beta$ has the law of the variable $X_N$ under the distribution $\proba_{(\beta,\delta_N)}$ with $\delta_N = \frac{h_N}{2}$ (circular Jacobi ensemble). Indeed,
    \begin{align*}
    \esper_\beta[\E^{zX_{N,h_N}}] &= \frac{1}{\esper_\beta[\E^{h_NX_N}]\,C_N(\beta)}\int_{\Tor^N}  \prod_{1\leq i<j\leq N}|\E^{\I\theta_i}-\E^{\I \theta_j}|^\beta\, \E^{(z+h_N)X_N} \DD{\theta_1}\cdots \!\DD{\theta_N} \\
    &=\frac{1}{\esper_\beta[\E^{h_NX_N}]\,C_N(\beta)}\int_{\Tor^N}  \prod_{1\leq i<j\leq N}|\E^{\I\theta_i}-\E^{\I \theta_j}|^\beta \prod_{i=1}^N |1-\E^{\I\theta_i}|^{h_N} \,\E^{zX_N} \DD{\theta_1}\cdots \!\DD{\theta_N} \\
    &= \frac{C_N(\beta,\delta_N)}{\esper_\beta[\E^{h_NX_N}]\,C_N(\beta)} \,\,\esper_{(\beta,\delta_N)}[\E^{zX_N}],
    \end{align*}
    and the multiplicative factor equals $1$, as can be seen by setting $z=0$.

    \item the sequence of parameters $(a_{N,\beta})_{N \in \N}$, which is related to $(h_N)_{N \in \N}$ by the equation
    \begin{align}
    a_{N,\beta} &= \esper_\beta[X_{N,h_N}] = \esper_{(\beta,\delta_N)}[X_N].\label{eq:critical_parameter}
\end{align}
In order to keep track of the setting in which the calculations are made, we add an index $\beta$; the notation $a_N$ without index will be used when $\beta=2$.

    \item the corresponding sequence of variances $(v_{N,\beta})_{N \in \N}$, defined by
    \begin{align*}
    v_{N,\beta}&= \var_\beta(X_{N,h_N}) = \var_{(\beta,\delta_N)}(X_N). 
    \end{align*}    
    Again, we shall use the notation $v_N$ without index when considering the random characteristic polynomials with $\beta=2$.
\end{itemize}

\noindent Our objective is to understand the order of magnitude of these quantities in various regimes. Let us remark that if $\Lambda_{N,\beta}(z)=\log \esper_\beta[\E^{zX_N}]$ is the log-Laplace transform of $X_N$ under $\proba_\beta$,
then 
$$\log \left(\esper_\beta[\E^{zX_{N,h_N}}] \right)= \log \left(\frac{\esper_\beta[\E^{(z+h_N)X_N}]}{\esper_\beta[\E^{h_NX_N}]}\right)  = \Lambda_{N,\beta}(z+h_N) - \Lambda_{N,\beta}(h_N),$$
so by taking the two first derivatives of this equation at $z=0$, we obtain $a_{N,\beta} = \Lambda_{N,\beta}'(h_N)$ and $v_{N,\beta}=\Lambda_{N,\beta}''(h_N)$. Therefore, the question above is related to the estimation of the log-Laplace transform and its derivatives, for values of the parameter $h_N$ in a large range.\medskip

In Subsection \ref{subsec:estimates_mv_beta2} we perform the analysis of the case $\beta=2$; the estimation of $a_N$ and $v_N$ then follows  from the asymptotic expansion of the polygamma functions. In Subsection \ref{subsec:comparison}  we prove a Comparison Theorem \ref{thm:comparison} which allows one to transfer the results for Haar-distributed random matrices to general circular $\beta$ ensembles. In Subsection \ref{subsec:balanced_sequences}, we use the estimates of the previous paragraphs in order to identify the different regimes of fluctuations of the variables $X_N$ (Propositions \ref{prop:moderate_regime} and \ref{prop:large_regime}).

\subsection{Estimates for Haar-distributed unitary matrices}\label{subsec:estimates_mv_beta2}
Until the end of this paragraph $\beta=2$, and we therefore remove the index $\beta$ from the quantities considered. Given a complex number $z$, we denote $\eps_z=\frac{z}{N}$. We start with the following estimate which will also be useful later in order to compute the Legendre--Fenchel transform of $\Lambda_N$.

\begin{proposition}[Asymptotics of the log-Laplace transform, case $\beta=2$]\label{prop:laplace_estimate_haar}
Suppose $\Re(z)>0$. Then
$$\Lambda_N(z) = \log \Psi(z) + \frac{z^2\,\log N}{4} - \frac{3z^2}{8} + \frac{N^2}{2}\,b(\eps_z) + O\!\left(\frac{1}{N}+\frac{|z|^2}{N^2}\right),$$
where 
$$b(\eps) = (1+\eps)^2\log(1+\eps) - 2\left(1+\frac{\eps}{2}\right)^2\,\log\!\left(1+\frac{\eps}{2}\right)=\frac{3\eps^2}{4}+\frac{\eps^3}{4}+O(\eps^4),$$
and $\Psi(z) = \frac{G(1+\frac{z}{2})^2}{G(1+z)}$. Moreover, 
$$\log \Psi(z) = -\frac{z^2\,\log(2z)}{4}+\frac{3z^2}{8}- \frac{1}{12}\,\log\!\left(\frac{z}{4}\right) + \zeta'(-1) + O\!\left(\frac{1}{|z|}\right).$$
\end{proposition}

\begin{proof}
We combine the exact formula
\begin{align*}
\Lambda_N(z) &= \log\!\left(\frac{G(1+N)\,G(1+N+z)\,G(1+\frac{z}{2})^2}{G(1+N+\frac{z}{2})^2\,G(1+z)}\right) \\ 
&= \log \Psi(z) + \log\!\left(\frac{G(1+N)\,G(1+N+z)}{G(1+N+\frac{z}{2})^2}\right)
\end{align*}
with the Stirling-like estimate \eqref{eq:stirling_barnes}.
\end{proof}
\medskip

We now examine the three first derivatives of $\Lambda_N$; the two first derivatives will provide asymptotic expansions of $a_N$ and $v_N$, whereas the third derivative will be used in Section \ref{sec:CLT} in the computation of the speed of convergence of certain central limit theorems. We shall use one of the reflection formula satisfied by the Barnes $G$-function:
$$\frac{d}{\!\DD{z}} \log G(1+z) = \frac{\log(2\pi)+1}{2} + z\,\psi_0(z) - z,$$
with $\psi_0(z) = \frac{d}{\!\DD{z}} (\log \Gamma(z))$ (the digamma function); see \cite[Equation (A.13)]{Vor87}. More generally we set $\psi_m(z)=\frac{d^{m+1}}{\!\DD{z}^{m+1}} (\log \Gamma(z))$. The asymptotic expansions of the first polygamma functions are:
\begin{align*}
\psi_0(z) &= \log z - \frac{1}{2z} -\frac{1}{12z^2} + O\!\left(\frac{1}{|z|^4}\right);\\
\psi_1(z) &= \frac{1}{z} + \frac{1}{2z^2}+ \frac{1}{6z^3}+O\!\left(\frac{1}{|z|^5}\right);\\
\psi_2(z) &= -\frac{1}{z^2} - \frac{1}{z^3} -\frac{1}{2z^4} + O\!\left(\frac{1}{|z|^6}\right).
\end{align*}
Indeed, the classical Stirling asymptotic expansion of $\log(\Gamma(1+z))$ can be obtained by using the Laplace method, therefore it can be differentiated term by term; see \cite[Chapter 19, Lemma 4 and Examples 7 and 13]{Zor02}.
As a consequence, if $\Re(z)>0$, then:
\begin{align*}
\Lambda_N'(z) &= (N+z)\,\psi_0(N+z) - \left(N+\frac{z}{2}\right) \psi_0\!\left(N+\frac{z}{2}\right)+\frac{z}{2}\,\psi_0\!\left(\frac{z}{2}\right) - z\,\psi_0(z)\\
&=\frac{z}{2} \log \!\left(\frac{N}{2z}\right) + N\big((1+\eps_z)\,\log(1+\eps_z) - (1+\eps_\frac{z}{2})\,\log (1+\eps_\frac{z}{2})\big) \\
&\quad+\frac{1}{12}\left(\frac{1}{N+\frac{z}{2}}-\frac{1}{z}-\frac{1}{N+z}\right)+O\!\left(\frac{1}{|z|^3}\right);\\
\Lambda_N''(z) &= \psi_0(N+z) - \frac{1}{2}\,\psi_0\!\left(N+\frac{z}{2}\right) + \frac{1}{2}\,\psi_0\!\left(\frac{z}{2}\right) - \psi_0(z) \\
&\quad+ (N+z)\,\psi_1(N+z)- \frac{1}{2}\left(N+\frac{z}{2}\right) \psi_1\!\left(N+\frac{z}{2}\right)+\frac{z}{4}\,\psi_1\!\left(\frac{z}{2}\right) - z\,\psi_1(z)\\
&= \frac{1}{2}\log\!\left(\frac{N}{2z}\right) + \left(\log(1+\eps_z) - \frac{1}{2}\log(1+\eps_{\frac{z}{2}})\right)\\
&\quad+\frac{1}{12}\left(\frac{1}{(N+z)^2}-\frac{1}{2(N+\frac{z}{2})^2} + \frac{1}{z^2}\right)+O\!\left(\frac{1}{|z|^4}\right);\\
\Lambda_N'''(z)&= 2\,\psi_1(N+z) - \frac{1}{2}\,\psi_1\!\left(N+\frac{z}{2}\right)+\frac{1}{2}\,\psi_1\!\left(\frac{z}{2}\right)-2\,\psi_1(z)\\
&\quad+ (N+z)\,\psi_2(N+z)- \frac{1}{4}\left(N+\frac{z}{2}\right) \psi_2\!\left(N+\frac{z}{2}\right)+\frac{z}{8}\,\psi_2\!\left(\frac{z}{2}\right) - z\,\psi_2(z)\\
&= -\frac{N^2}{z(N+z)(2N+z)}+O\!\left(\frac{1}{|z|^3}\right).
\end{align*}
So by taking the two first identities with $z=h_N=2\delta_N$ we obtain:

\begin{proposition}[Estimates of the mean and the variance, case $\beta=2$]\label{prop:mean_variance_beta2}
Under the Hua--Pickrell measures $\proba_{(2,\delta_N)}$, uniformly for $\delta_N>0$, we have:
\begin{align}
\esper_{(2,\delta_N)}[X_N] &= \delta_N\,\log\!\left(\frac{N}{4\delta_N}\right) + N\left(\left(1+2\eps_{\delta_N}\right)\log\!\left(1+2\eps_{\delta_N}\right) - \left(1+\eps_{\delta_N}\right)\log\left(1+\eps_{\delta_N}\right)\right)\notag\\
&\quad+\frac{1}{12}\left(\frac{1}{N+\delta_N} - \frac{1}{2\delta_N}-\frac{1}{N+2\delta_N}\right)+O\!\left(\frac{1}{(\delta_N)^3}\right);\label{eq:mean_huapickrell}\\
\var_{(2,\delta_N)}(X_N) &= \frac{1}{2}\log\!\left(\frac{N}{4\delta_N}\right) + \left(\log\left(1+2\eps_{\delta_N}\right) - \frac{1}{2}\,\log\left(1+\eps_{\delta_N}\right)\right)\notag\\
&\quad+\frac{1}{12}\left( \frac{1}{(N+2\delta_N)^2} - \frac{1}{2(N+\delta_N)^2}+ \frac{1}{4(\delta_N)^2}\right)+O\!\left(\frac{1}{(\delta_N)^4}\right).\label{eq:variance_huapickrell}
\end{align}
\end{proposition}

\noindent These asymptotic expansions of $a_N$ and $v_N$ will be used many times in the sequel. Note that if $h_N$ satisfies $1 \lesssim h_N \ll N$, then the terms of the expansion \eqref{eq:mean_huapickrell} are ordered by decreasing magnitude. In this setting the second line of Equation \eqref{eq:mean_huapickrell} is a $O((h_N)^{-1})$, and the second line of Equation \eqref{eq:variance_huapickrell} is a $O((h_N)^{-2})$.

\subsection{Comparison of unitary ensembles}\label{subsec:comparison}
Let us now consider the case of a general parameter $\beta>0$. In the following we denote
$$ \lambda_N = \frac{h_N}{\beta}=\frac{\delta_N}{\beta'}.$$
It turns out that the mean $a_{N,\beta}$ and the variance $v_{N,\beta}$ admit nice expressions if we introduce this new parameter $\lambda_N$. Let us remark that Proposition \ref{prop:mean_variance_beta2} expresses $a_N=a_{N,2}$ and $v_N=v_{N,2}$ as functions of $\lambda_N$: it suffices to replace $\delta_N$ by $\lambda_N$ in the formul{\ae} of the proposition. In the sequel, if $\beta \neq 2$, each time we write $a_N$ or $v_N$ \emph{without the index $\beta$}, we mean the aforementioned functions of the parameter $\lambda_N$, and we shall see that they are closely related to the parameters $a_{N,\beta}$ and $v_{N,\beta}$, which are also functions of $\lambda_N$. \medskip

The exact formula for the logarithm of the Laplace transform of $X_N$ under $\proba_{\beta}$ is
$$\Lambda_{N,\beta}(z) = \log \!\left(\esper_{\beta}[\E^{zX_N}]\right) = \sum_{k=0}^{N-1}  \left(\ell(\beta'k+1)+\ell(\beta'k+1+z) - 2\,\ell\!\left(\beta'k+1+\frac{z}{2}\right)\right),$$
where $\ell(z)=\log\Gamma(z)$. By the Binet formula (see \emph{e.g.} \cite{Sas99}), for $\Re(z)>0$,
$$\ell(z+1) = \left(z+\frac{1}{2}\right)\log z - z +\frac{1}{2}\log(2\pi)+\int_0^{\infty} \E^{-sz}\varphi(s)\DD{s}$$
with $\varphi(s)=\frac{1}{s}(\frac{1}{2}-\frac{1}{s}+\frac{1}{\E^s-1})$. Therefore $\Lambda_{N,\beta}(z)=m(z)+g_{N,\beta}(z)+k_{N,\beta}(z)$, with
\begin{align*}
m(z) &= \ell(1+z) - 2\ell\!\left(1+\frac{z}{2}\right);\\
g_{N,\beta}(z) &=\int_0^{\infty} \frac{1-\E^{-s\beta'(N-1)}}{1-\E^{-s\beta'}}\,\left(1-\E^{-\frac{sz}{2}}\right)^2\,\E^{-s\beta'}\,\varphi(s)\DD{s};\\
k_{N,\beta}(z)&=\sum_{k=1}^{N-1} \left(\kappa(\beta'k) + \kappa (\beta'k+z)-2\, \kappa\!\left(\beta'k+\frac{z}{2}\right)\right)
\end{align*}
with $\kappa(y) = (y+\frac{1}{2})\log y$ on the last line.
We set $\phi_\beta(s)=\varphi(s)-\beta'^2\,\varphi(s\beta')$, $g_N(z)=g_{N,2}(z)$ and $k_N(z)=k_{N,2}(z)$. Our main tool will be the following identity:

\begin{theorem}[Comparison between Haar ensembles and circular $\beta$ ensembles]\label{thm:comparison}
For any $\beta>0$, we have:
\begin{align*}
\Lambda_{N,\beta}(z) &= \beta'\,\Lambda_N\!\left(\frac{z}{\beta'}\right) \\
&\quad+ \frac{\beta'-1}{2}\left(2\ell\!\left(N+\frac{z}{2\beta'}\right) - \ell(N) - \ell\!\left(N+\frac{z}{\beta'}\right) \right) + m(z) - \frac{\beta'+1}{2}\, m\!\left(\frac{z}{\beta'}\right) \\
&\quad+ \int_0^{\infty} \frac{1-\E^{-s\beta'(N-1)}}{1-\E^{-s\beta'}}\,\left(1-\E^{-\frac{sz}{2}}\right)^2\,\E^{-s\beta'}\,\phi_\beta(s)\,\DD{s}.
\end{align*}
\end{theorem}

\begin{proof}
A straightforward calculation yields:
\begin{align*}
k_{N,\beta}(z) - \beta'\,k_N\!\left(\frac{z}{\beta'}\right) &= \frac{\beta'-1}{2}\,\sum_{k=1}^{N-1} \log \!\left(\frac{\left(k+\frac{z}{2\beta'}\right)^2}{k(k+\frac{z}{\beta'})}\right) \\
&= \frac{\beta'-1}{2}\,\left(2\ell\!\left(N+\frac{z}{2\beta'}\right) - \ell(N) - \ell\!\left(N+\frac{z}{\beta'}\right) + m\!\left(\frac{z}{\beta'}\right)\right).
\end{align*}
If we add $m(z)-\beta'\,m(\frac{z}{\beta'})+g_{N,\beta}(z)-\beta'\,g_N(\frac{z}{\beta'})$ to this identity, then we obtain the formula of the proposition since a change of variables gives:
\begin{align*}
\beta'\,g_N\!\left(\frac{z}{\beta'}\right)&=\beta'\,\int_0^{\infty} \frac{1-\E^{-s(N-1)}}{1-\E^{-s}}\,\left(1-\E^{-\frac{sz}{2\beta'}}\right)^2\,\E^{-s}\,\varphi(s)\DD{s}\\
&= (\beta')^2\,\int_0^{\infty} \frac{1-\E^{-t\beta'(N-1)}}{1-\E^{-t\beta'}}\,\left(1-\E^{-\frac{tz}{2}}\right)^2\,\E^{-t\beta'}\,\varphi(t\beta')\DD{t}.\qedhere
\end{align*}
\end{proof}
\medskip

In the comparison theorem above, all the terms depend smoothly on $z$, and the only quantity that is somewhat difficult to analyse is the integral. 

\begin{lemma}\label{lem:laplace_method}
 Suppose $\beta \neq 2$. We set 
$$G_{N,\beta}(z) = \frac{12\beta'}{1-\beta'^2}\int_0^\infty \frac{1-\E^{-s\beta'(N-1)}}{1-\E^{-s\beta'}}\,\left(1-\E^{-\frac{sz}{2}}\right)^2\,\E^{-s\beta'}\phi_\beta(s)\DD{s}.$$
 Then
 \begin{align*}
(G_{N,\beta})'(h_N)&=\frac{1}{\beta'}\left(\frac{1}{2\lambda_N} + \frac{1}{2\lambda_N+N}-\frac{1}{\lambda_N+N}\right) + O_\beta\!\left(\frac{1}{(\lambda_N)^2}\right);\\
 (G_{N,\beta})''(h_N)&= \frac{1}{\beta'^2}\left(-\frac{1}{4(\lambda_N)^2} - \frac{1}{(2\lambda_N+N)^2} + \frac{1}{2(\lambda_N + N)^2}\right)+ O_\beta\!\left(\frac{1}{(\lambda_N)^3}\right);\\
 (G_{N,\beta})'''(h_N+\I\xi)&=O_\beta\!\left(\frac{1}{(\lambda_N)^3}\right).
 \end{align*}
\end{lemma}
 
\begin{proof}
Note that $\varphi(0)=\frac{1}{12}$; 
 therefore $\phi_\beta(0)=\frac{1-\beta'^2}{12}$ 
and if $\eta_\beta(s) = \frac{s\beta'\,\phi_\beta(s)}{(\E^{s\beta'}-1)\,\phi_\beta(0)}$, then $\eta_\beta(0)=1$. 

\noindent We have 
\begin{align*}
G_{N,\beta}(z)&=\int_0^\infty \eta_\beta(s)\,(1-\E^{-s\beta'(N-1)})\left(1-\E^{-\frac{sz}{2}}\right)^2\frac{\DD{s}}{s}\\
(G_{N,\beta})'(z) &= \int_0^\infty \eta_\beta(s)\,\left(1-\E^{-s\beta'(N-1)}\right)\left(\E^{-\frac{sz}{2}}-\E^{-sz}\right)\DD{s}.
\end{align*}
By the Laplace method (see \emph{e.g.} \cite[Section 19.2, p. 619]{Zor02}), when $\lambda$ is a positive real parameter,
\begin{align*}
\int_0^\infty \eta_\beta(s)\,\E^{-\lambda s}\DD{s} 
&= \frac{\eta_\beta(0)}{\lambda}+ O\!\left(\frac{1}{\lambda^2}\right);\\
\int_0^\infty s\,\eta_\beta(s)\,\E^{-\lambda s}\DD{s} 
&= \frac{\eta_\beta(0)}{\lambda^2} +O\!\left(\frac{1}{\lambda^3}\right) ;\\
\int_0^\infty s^2\,\eta_\beta(s)\,\E^{-\lambda s}\DD{s}
&= \frac{2\eta_\beta(0)}{\lambda^3} + O\!\left(\frac{1}{\lambda^4}\right).
\end{align*}
If we replace the remainders $O(\lambda^{-k})$ by $O((\Re(\lambda))^{-k})$, then these estimates still hold for $\lambda$ complex number with a positive real part. As a consequence of the first formula above, if we expand in the integral $(G_{N,\beta})'(z)$ the product $(1-\E^{-s\beta'(N-1)})(\E^{-\frac{sz}{2}}-\E^{-sz})$, then we obtain for $z=h_N$:
\begin{align*}
(G_{N,\beta})'(h_N) &=  \frac{\eta_\beta(0)}{h_N} +\frac{\eta_\beta(0)}{\beta'(N-1)+h_N}-\frac{\eta_\beta(0)}{\beta'(N-1)+\frac{h_N}{2}}+O_\beta\!\left(\frac{1}{(h_N)^2}\right) \\
&= \frac{1}{h_N} + \frac{1}{\beta'N+h_N}-\frac{1}{\beta'N+\frac{h_N}{2}} +O_\beta\!\left(\frac{1}{(h_N)^2}\right).
\end{align*}
Similarly we have
\begin{align*}
(G_{N,\beta})''(z) &= \int_0^\infty s\,\eta_\beta(s)\,\left(1-\E^{-s\beta'(N-1)}\right)\left(\E^{-sz}-\frac{1}{2}\,\E^{-\frac{sz}{2}}\right)\DD{s} ;\\
(G_{N,\beta})''(h_N) &= -\frac{1}{(h_N)^2} - \frac{1}{(\beta'N + h_N)^2} + \frac{1}{2(\beta'N + \frac{h_N}{2})^2}+ O_\beta\!\left(\frac{1}{(h_N)^3}\right).
\end{align*}
Finally,
\begin{align*}
(G_{N,\beta})'''(z) &= \int_0^\infty s^2\,\eta_\beta(s)\,\left(1-\E^{-s\beta'(N-1)}\right)\left(\frac{1}{4}\,\E^{-\frac{sz}{2}}-\E^{-sz}\right)\DD{s}; \\
(G_{N,\beta})'''(h_N+\I \xi) &= \frac{2}{(h_N+\I\xi)^3} + \frac{2}{(h_N+\I\xi+\beta'N)^3} - \frac{1}{2(\frac{h_N+\I\xi}{2}+\beta'N)^3} + O_\beta\!\left(\frac{1}{(h_N)^4}\right)\\
&=O_\beta\!\left(\frac{1}{(h_N)^3}\right)\!.
\end{align*}
Replacing $h_N$ by $2\beta'\lambda_N$ yields the results announced.
\end{proof}
\medskip

We can now state the analogues of Propositions \ref{prop:laplace_estimate_haar} and \ref{prop:mean_variance_beta2} for a general parameter $\beta>0$.

\begin{proposition}[Asymptotics of the log-Laplace transform, case $\beta\neq 2$]\label{prop:laplace_estimate_beta}
Suppose $\Re(z)>0$. Then,
\begin{align*}
\Lambda_{N,\beta}(z) &= \log \Psi_\beta(z) + \frac{z^2\,\log N}{4\beta'} - \frac{3z^2}{8\beta'} + \frac{N^2}{2}\,\beta'\, b(\eps_{\frac{z}{\beta'}}) + \frac{N}{2}\, (\beta'-1) \,c(\eps_{\frac{z}{\beta'}}) + O_\beta\!\left(\frac{1}{N}+\frac{|z|^2}{N^2}\right)
\end{align*}
where 
\begin{align*}
\log \Psi_\beta(z)&=\beta'\log \Psi\!\left(\frac{z}{\beta'}\right) + m(z) - \frac{\beta'+1}{2}\, m\!\left(\frac{z}{\beta'}\right)  
+ \frac{1-\beta'^2}{12\beta'}\int_0^{\infty} \frac{\left(1-\E^{-\frac{sz}{2}}\right)^2}{s}\,\eta_\beta(s)\DD{s};\\
c(\eps) &= 2\left(1+\frac{\eps}{2}\right) \log\!\left(1+\frac{\eps}{2}\right) -(1+\eps)\log(1+\eps)=-\frac{\eps^2}{4}+\frac{\eps^3}{8}+O(\eps^4).
\end{align*}
Moreover, 
\begin{align*}
\log \Psi_\beta(z)&=-\frac{z^2\,\log(\frac{2z}{\beta'})}{4\beta'}+\frac{3z^2}{8\beta'} + \frac{((\beta'-1)\log 2)\,z}{2\beta'} + \frac{2\beta'-3}{12}\,\log\!\left(\frac{z}{4\beta'}\right) + \beta'\zeta'(-1)\\
& \quad + \frac{\beta'-1}{4}\log (2\pi) - \frac{1}{2}\log \beta'+ \frac{1-\beta'^2}{12\beta'} \int_0^{\infty} \frac{\left(1-\E^{-\frac{sz}{2}}\right)^2}{s}\,\eta_\beta(s)\DD{s}+ O_\beta\!\left(\frac{1}{|z|}\right).
\end{align*}
In this second formula, when $z=h_N$ is a large positive real number, the integral of the second line is equal to $\log h_N + A_\beta+ O_\beta(\frac{1}{h_N})$ for some explicit constant $A_\beta$.
\end{proposition}

\begin{proof}
By combining the first part of Proposition \ref{prop:laplace_estimate_haar}, the Comparison Theorem \ref{thm:comparison} and the Stirling approximation of $\ell(z)$ for $z$ large, we obtain:
\begin{align*}
\Lambda_{N,\beta}(z) &= \text{right-hand side of the first formula} -\int_0^{\infty} \E^{-s\beta'N}\,\frac{\left(1-\E^{-\frac{sz}{2}}\right)^2}{1-\E^{-s\beta'}}\,\phi_\beta(s)\DD{s},
\end{align*}
so it suffices to check that the integral is a $O_\beta(\frac{|z|^2}{N^2})$. However by the Taylor integral formula,
$$\left|\frac{1-\E^{-\frac{sz}{2}}}{\frac{sz}{2}}\right|=\left|\int_0^1 \E^{-\frac{usz}{2}}\DD{u}\right| \leq 1,$$
so with the same notations as in proof of Lemma \ref{lem:laplace_method}
$$\left|\int_0^{\infty} \E^{-s\beta'N}\,\frac{\left(1-\E^{-\frac{sz}{2}}\right)^2}{1-\E^{-s\beta'}}\,\phi_\beta(s)\DD{s}\right| \lesssim |z|^2 \int_0^\infty \E^{-s\beta'(N-1)}\,s\,\eta_\beta(s)\DD{s} = O_\beta\!\left(\frac{|z|^2}{N^2}\right).$$
We now combine the second part of Proposition \ref{prop:laplace_estimate_haar} and the Stirling approximation
$$m(z) = (z+1) \log 2 -\frac{1}{2} \log (2\pi z) + O\!\left(\frac{1}{|z|}\right)$$
in order to compute an approximation of $\log \Psi_\beta(z)$; we obtain the second formula of the proposition. Let us finally estimate the integral $F_\beta(z)=\int_0^\infty (1-\E^{-\frac{sz}{2}})^2\,\frac{\eta_\beta(s)}{s}\DD{s}$. We remark that
\begin{align*}
F_\beta'(z)&=\int_0^{\infty} \left(\E^{-\frac{sz}{2}}-\E^{-sz}\right)\,\eta_\beta(s)\DD{s} ;\\
F_\beta'(z) - \frac{1}{z} &=\int_0^{\infty} \left(\E^{-\frac{sz}{2}}-\E^{-sz}\right)\,(\eta_\beta(s)-\eta_\beta(0))\DD{s}\\
&= \frac{1}{z}\int_0^\infty\left(2\E^{-\frac{sz}{2}}-\E^{-sz}\right)\eta_\beta'(s)\DD{s} = O_\beta\!\left(\frac{1}{(\Re(z))^2}\right).
\end{align*}
Therefore, for $z=h_N$ large positive real number,
\begin{align*}
F_\beta(h_N) &= \log h_N + F_\beta(1) + \int_{t=1}^{h_N} \int_{s=0}^\infty\frac{2\E^{-\frac{st}{2}}-\E^{-st}}{t}\, \eta_\beta'(s)\DD{s}\DD{t} \\
&=\log h_N + F_\beta(1) + \int_{t=1}^\infty \int_{s=0}^\infty \frac{2\E^{-\frac{st}{2}}-\E^{-st}}{t}\, \eta_\beta'(s)\DD{s} \DD{t} + O_\beta\!\left(\frac{1}{h_N}\right).
\end{align*}
This proves the last part of the proposition, with 
\begin{equation*}
    A_\beta=\int_{s=0}^\infty(1-\E^{-\frac{s}{2}})^2\,\frac{\eta_\beta(s)}{s}\DD{s} + \int_{t=1}^\infty \int_{s=0}^\infty \frac{2\E^{-\frac{st}{2}}-\E^{-st}}{t}\, \eta_\beta'(s)\DD{s}\DD{t}. \qedhere
\end{equation*}
\end{proof}

\begin{proposition}[Estimates of the mean and the variance, case $\beta \neq 2$]\label{prop:mean_variance_jacobi} Under the circular Jacobi $(\beta,\delta_N)$ distributions, uniformly for $\delta_N>0$, we have:
 \begin{align*}
 \esper_{(\beta,\delta_N)}[X_N] &=a_{N,\beta} = a_N + \frac{\beta'-1}{2\beta'} \left(\log 2 + \log(1+\eps_{\lambda_N}) - \log(1+2\eps_{\lambda_N})\right) \\
 &\quad+ \frac{(1-\beta')(1-2\beta')}{12(\beta')^2}\left(\frac{1}{2\lambda_N}+\frac{1}{N+2\lambda_N}-\frac{1}{N+\lambda_N }\right) + O_\beta\!\left(\frac{1}{(\lambda_N)^2}\right)
 \end{align*}
 and
 \begin{align*}
 \var_{(\beta,\delta_N)}(X_N) &= v_{N,\beta} = \frac{1}{\beta'}\,v_N +\frac{\beta'-1}{2(\beta')^2}\left(\frac{1}{2N+2\lambda_N}  -\frac{1}{N+2\lambda_N} \right) \\
 &+\frac{(1-\beta')(1-2\beta')}{12(\beta')^3} \left(\frac{1}{2(N+\lambda_N)^2} - \frac{1}{(N+2\lambda_N)^2} -\frac{1}{4(\lambda_N)^2}\right) + O_\beta\!\left(\frac{1}{(\lambda_N)^3}\right),
 \end{align*}
 where $\lambda_N = \frac{\delta_N}{\beta'}$, and where $a_{N}$ and $v_N$ correspond to the case $\beta=2$ with the same parameter $\lambda_N$ and are estimated by Equations \eqref{eq:mean_huapickrell} and \eqref{eq:variance_huapickrell} (with $\lambda_N=\delta_N$ in these equations).
\end{proposition}

\begin{remark}
We insist on the fact that the proposition above relates $a_{N,\beta}=\esper_\beta[X_{N,\beta\lambda_N}]$ and $v_{N,\beta}=\var_\beta(X_{N,\beta\lambda_N})$ to $a_N = \esper_2[X_{N,2\lambda_N}]$ and $v_N = \var_2(X_{N,2\lambda_N})$; the tilting parameter $h_N$ is not the same for the circular $\beta$ ensemble and for the Haar ensemble, but the scaled tilting parameter $\lambda_N$ is the same.
\end{remark}

\begin{proof}
Since $a_{N,\beta} = \Lambda_{N,\beta}'(h_N)$ and $v_{N,\beta} = \Lambda_{N,\beta}''(h_N)$, the Comparison Theorem \ref{thm:comparison} yields the exact formul{\ae}
\begin{align*}
a_{N,\beta} &= a_N + \psi_0(h_N)-\psi_0\!\left(\frac{h_N}{2}\right)+\frac{\beta'-1}{2h_N}+\frac{\beta'+1}{2\beta'}\left(\psi_0\!\left(\frac{h_N}{2\beta'}\right)-\psi_0\!\left(\frac{h_N}{\beta'}\right)\right)\\
&\quad+\frac{\beta'-1}{2\beta'}\left(\psi_0\!\left(N+\frac{h_N}{2\beta'}\right)-\psi_0\!\left(N+\frac{h_N}{\beta'}\right)\right)+\frac{1-\beta'^2}{12\beta'}\,(G_{N,\beta})'(h_N);\\
v_{N,\beta}&= \frac{1}{\beta'}\,v_N+ \psi_1(h_N) - \frac{1}{2}\,\psi_1\!\left(\frac{h_N}{2}\right) -\frac{\beta'-1}{2(h_N)^2}+\frac{\beta'+1}{2\beta'^2}\left(\frac{1}{2}\,\psi_1\!\left(\frac{h_N}{2\beta'}\right)-\psi_1\!\left(\frac{h_N}{\beta'}\right)\right)\\
&\quad+\frac{\beta'-1}{2\beta'^2}\left(\frac{1}{2}\,\psi_1\!\left(N+\frac{h_N}{2\beta'}\right)-\psi_1\!\left(N+\frac{h_N}{\beta'}\right)\right)+\frac{1-\beta'^2}{12\beta'}\,(G_{N,\beta})''(h_N).
\end{align*}
The result follows immediately by using Lemma \ref{lem:laplace_method}, the asymptotic expansions of the poly\-gamma functions and the relation $h_N=2\beta'\lambda_N$.
\end{proof}

\subsection{Balanced sequences of parameters and the regimes of fluctuations}\label{subsec:balanced_sequences}
In the previous paragraphs, we have computed $a_{N,\beta}$ and $v_{N,\beta}$ in terms of the tilting parameter $h_N$ (or, of the rescaled tilting parameter $\lambda_N$). Conversely, given a sequence of positive parameters $(a_{N,\beta})_{N \in \N}$, we can recover the corresponding sequence $(h_N)_{N \in \N}$ if $0< a_{N,\beta} < N \log 2$ for any $N$: indeed, the function $h_N \mapsto \Lambda_{N,\beta}'(h_N)=\esper_\beta[X_{N,h_N}]$ is an increasing bijection from $\R_+$ to $[0,N\log2 )$, as
$$N \log 2 = \max \{\Re \log \det(I_N-U_N),\,\,U_N \in \mathrm{U}(N)\}.$$
Let us now compare the growths of the two sequences $(\lambda_N)_{N \in \N}$ and $(a_{N,\beta})_{N\in \N}$.

\begin{lemma}\label{lem:regime_sequences}
Fix $\beta>0$. We have the following equivalences:
\begin{enumerate}
    \item $\log N \lesssim a_{N,\beta}$ if and only if\, $1 \lesssim \lambda_N$.
    \item $\log N \ll a_{N,\beta}$ if and only if\, $1 \ll \lambda_N$.
    \item $\limsup_{N \to \infty} \frac{a_{N,\beta}}{N\,\log 2}<1$ if and only if\, $\limsup_{N \to \infty} \frac{\lambda_N}{N}<+\infty$.
\end{enumerate}
\end{lemma}

\begin{proof}
By Equation \eqref{eq:mean_huapickrell} and Proposition \ref{prop:mean_variance_jacobi}, we have
\begin{align*}
a_{N,\beta} &= a_N + O_\beta\!\left(1+\frac{1}{\lambda_N}+\log\!\left(1+\frac{2\lambda_N}{N}\right)\right) \\
&= N\,\theta\!\left(\frac{\lambda_N}{N}\right) + O_\beta\!\left(1+\frac{1}{\lambda_N}+\log\!\left(1+\frac{2\lambda_N}{N}\right)\right)
\end{align*}
where $\theta(x) = (1+2x)\log(1+2x)-(1+x)\log(1+x)-x\log(4x)$. 
Suppose that $\lambda_N \in [c,C]$ with $c>0$. Then this interval being fixed, the $O_\beta(\cdot)$ above is a $O_{\beta,c,C}(1)$ and on the other hand we have the following asymptotic expansion of $\theta$ in a neighborhood of $0$:
$\theta(x) = -x\log x+O(x)$.
Therefore $a_{N,\beta} = \lambda_N \log N +O_{\beta,c,C}(1)$, so:
\begin{align*}
\left(\forall N,\,\,c \leq \lambda_N \leq C\right) &\Rightarrow \left(c \leq \liminf_{N \to +\infty}\frac{a_{N,\beta}}{\log N} \right);\\
\left(\forall N,\,\,c \leq \lambda_N \leq C\right) &\Rightarrow \left(\limsup_{N \to +\infty}\frac{a_{N,\beta}}{\log N} \leq C \right).
\end{align*}
 Now as $a_{N,\beta}$ and $\lambda_N$ are simultaneously increasing, we can remove the restriction $\lambda_N \leq C$ in the first implication and the restriction $c \leq \lambda_N$ in the second implication. We therefore obtain
\begin{align}
\left(c \leq \liminf_{N \to \infty} \lambda_N \right) &\Rightarrow \left(c \leq \liminf_{N \to \infty} \frac{a_{N,\beta}}{\log N} \right);\label{eq:implication_liminf}\\
\left(\limsup_{N \to \infty} \lambda_N \leq C \right) &\Rightarrow \left( \limsup_{N \to \infty} \frac{a_{N,\beta}}{\log N} \leq C\right).\label{eq:implication_limsup}
\end{align}
Equation \eqref{eq:implication_limsup} can be used to prove that the implication \eqref{eq:implication_liminf} is in fact an equivalence for any $c>0$. Indeed suppose that $\liminf_{N \to \infty}  \frac{a_{N,\beta}}{\log N} \geq c$. Then for any $\eps>0$ we have at least $\liminf_{N \to \infty} \lambda_N \geq (1-\eps)c $. Otherwise we could extract a subsequence $(\lambda_{N_i})_{i \in \N}$ with $\lambda_{N_i}<(1-\eps)c$ for all indices $i$, and by Equation \eqref{eq:implication_limsup} applied to this subsequence, we would obtain $\limsup_{i \to \infty} \frac{a_{N_i,\beta}}{\log N_i} \leq (1-\eps)c$. This would be a contradiction. Thus $\liminf_{N \to \infty} \lambda_N \geq (1-\eps)c $ for any $\eps>0$ and $\liminf_{N \to \infty} \lambda_N \geq c$.
\medskip

\noindent The equivalence   
$$\left(c \leq \liminf_{N \to \infty} \lambda_N \right) \iff \left(c \leq \liminf_{N \to \infty} \frac{a_{N,\beta}}{\log N} \right)$$
for any $c>0$ implies immediately the two first items of the Proposition. For the third item, let us suppose first that $\frac{\lambda_N}{N}$ is bounded from above by a constant $C$. If $\lambda_N \leq C$, then we are in the same situation as above, and $a_{N,\beta} \leq (C+o_\beta(1)) \log N$; \emph{a fortiori}, $\frac{a_{N,\beta}}{N\,\log 2}<1$ for $N$ large enough. On the other hand if $C \leq \lambda_N \leq C N$, then the $O_\beta(\cdot)$ in the estimate of $a_{N,\beta}$ written at the beginning of this proof is a $O_{\beta,C}(1)$, so:
$$\frac{a_{N,\beta}}{N} = \theta\!\left(\frac{\lambda_N}{N}\right) + O_{\beta,C}\!\left(\frac{1}{N}\right)$$
and $\limsup_{N \to \infty} \frac{a_{N,\beta}}{N} \leq \theta(C)<\theta(+\infty)=\log 2$. This proves one implication and the converse implication has an analogous proof.
\end{proof}
\medskip

In the following we shall consider sequences $(a_{N,\beta})_{N \in \N}$ such that
$$\liminf_{N \to +\infty} \lambda_N > 0 \qquad;\qquad \lim_{N \to +\infty} v_{N,\beta}=+\infty,$$
$\lambda_N = \frac{h_N}{\beta}$ being associated to $a_{N,\beta}$ by Equation \eqref{eq:critical_parameter}.
We call such sequences \emph{balanced} and being balanced will be a sufficient condition in order to obtain an asymptotic equivalent of the probability $\proba_\beta[X_N \geq a_{N,\beta}]$. This corresponds to the regime of \emph{moderate deviations}, which is identified by the following:

\begin{proposition}[Regime of moderate deviations]\label{prop:moderate_regime}
Given a sequence $(a_{N,\beta})_{N \in \N}$, the following conditions are equivalent:
\begin{enumerate}
    \item The sequence $(a_{N,\beta})_{N \in \N}$ is balanced.
    \item We have $\liminf_{N \to \infty} \lambda_N >0$ and $\lim_{N \to \infty} \frac{\lambda_N}{N}=0$.
    \item We have $\liminf_{N \to \infty} \frac{a_{N,\beta}}{\log N} > 0$ and $\lim_{N \to \infty} \frac{a_{N,\beta}}{N}=0$.
\end{enumerate} 
\end{proposition}

\begin{proof}
We deal with  the case $\beta=2$ (so, $\lambda_N=\delta_N$); the general case follows by similar arguments, thanks to the Comparison Theorem \ref{thm:comparison} and to Proposition \ref{prop:mean_variance_jacobi}. Notice first that if $\liminf_{N \to \infty} \lambda_N >0$, then by Equation \eqref{eq:variance_huapickrell},
$$v_N = \frac{1}{2} \log\!\left(1+\frac{1}{4\,\eps_N(1+\eps_N)}\right)+ O(1),$$
where $\eps_N = \frac{\lambda_N}{N}$. Therefore, $v_N$ goes to $+\infty$ if and only if $\eps_N$ goes to $0$. This proves the equivalence between the two first items.\medskip

\noindent Let us now prove the equivalence between the two last items. We already know that $\liminf_{N \to \infty} \lambda_N > 0$ if and only if $\liminf_{N \to \infty} \frac{a_{N}}{\log N} > 0$ (this is the first item of the previous proposition). In this setting, we have shown above that
$$\frac{a_N}{N} = \theta\!\left(\frac{\lambda_N}{N}\right) + O\!\left(\frac{1}{N}\right);$$
therefore, $\frac{a_N}{N}$ goes to $0$ if and only if $\frac{\lambda_N}{N}$ goes to $0$.
\end{proof}

\begin{proposition}[Variances in the regime of moderate deviations]\label{prop:log_N_aN}
Consider a balanced sequence $(a_{N,\beta})_{N \in \N}$. We have:
$$\beta\,v_{N,\beta} = \log\!\left(\frac{N}{\lambda_N}\right) +O_\beta(1)=\log\!\left(\frac{N}{a_{N,\beta}}\right) + O_\beta\!\left(\log \log \!\left(\frac{N}{a_{N,\beta}}\right) \right).$$
\end{proposition}

\begin{proof}
We first treat the case $\beta=2$. Knowing that $1\lesssim \lambda_N \ll N$, our usual estimates \eqref{eq:mean_huapickrell} and \eqref{eq:variance_huapickrell} yield:
\begin{align*} 
\frac{a_N}{N} &= \frac{\lambda_N}{N}\left(\log\!\left(\frac{N}{\lambda_N}\right) + O(1) \right);\\
2\,v_N &= \log\!\left(\frac{N}{\lambda_N}\right) + O(1).
\end{align*}
Taking the logarithm of the first equation shows that $\log (\frac{N}{\lambda_N})$ and $\log (\frac{N}{a_N})$ are asymptotically equivalent and we then have:
$$2\,v_N = \log\!\left(\frac{N}{\lambda_N}\right) +O(1)=\log\!\left(\frac{N}{a_N}\right) + O\!\left(\log \log \!\left(\frac{N}{a_N}\right) \right).$$
For $\beta\neq 2$ we have by Proposition \ref{prop:mean_variance_jacobi}:
\begin{align*}
\frac{a_{N,\beta}}{N} &= \frac{a_N + O_\beta(1)}{N} = \frac{a_N}{N}\left(1+O_\beta\!\left(\frac{1}{a_N}\right)\right); \\
\beta\,v_{N,\beta} &= 2\,v_N + O_\beta(1)
\end{align*}
if $(a_{N,\beta})_{N \in \N}$ is a balanced sequence. Therefore, 
\begin{align*}
\beta\,v_{N,\beta} = \log\!\left(\frac{N}{\lambda_N}\right) +O_\beta(1) &= \log\!\left(\frac{N}{a_N}\right) + O_\beta\!\left(\log \log \!\left(\frac{N}{a_N}\right) \right)\\
&=\log\!\left(\frac{N}{a_{N,\beta}}\right) + O_\beta\!\left(\log \log \!\left(\frac{N}{a_{N,\beta}}\right) \right).\qedhere
\end{align*}
\end{proof}
\medskip

If we do not have $\frac{a_{N,\beta}}{N} \to 0$, then we fall in the regime of \emph{large deviations} which is covered by Theorem \ref{thm:super_large}. Let us summarise the asymptotic estimates which will be useful in this setting:

\begin{proposition}[Regime of large deviations]\label{prop:large_regime}
The following assertions are equivalent:
$$\liminf_{N \to \infty} \left(\frac{a_{N,\beta}}{N} \right)> 0\quad\iff\quad \liminf_{N \to \infty} \left(\frac{\lambda_N}{N} \right) > 0 \quad \iff \quad \limsup_{N \to \infty} v_{N,\beta} <+\infty.$$
Then these quantities are related by the following formulas:
\begin{align*}
\frac{a_{N,\beta}}{N} &=  \theta\!\left(\frac{\lambda_N}{N}\right) +  O_\beta\!\left(\frac{1}{N}\right);\\
v_{N,\beta} &=\frac{1}{\beta}\,\log \left(1+\frac{1}{4\frac{\lambda_N}{N}(1+\frac{\lambda_N}{N})}\right)+ O_\beta\!\left(\frac{1}{N}\right).
\end{align*}
\end{proposition}

\begin{proof}
The case $\beta=2$ ($\beta'=1$) follows immediately from the estimates \eqref{eq:mean_huapickrell} and \eqref{eq:variance_huapickrell}; in both cases the two first terms of these asymptotic expansions become of the same order of magnitude and their combination yield the formul{\ae} above. For $\beta \neq 2$ our results of comparison (Theorem \ref{thm:comparison} and Proposition \ref{prop:mean_variance_jacobi}) give in the regime $\lambda_N \approx N$:
\begin{align*}
a_{N,\beta} &= a_N + O_\beta(1);\\
v_{N,\beta} &= \frac{1}{\beta'}\,v_N+ O_\beta\!\left(\frac{1}{N}\right),
\end{align*}
whence the result.
\end{proof}
\bigskip

\section{Central limit theorems and their speed of convergence}\label{sec:CLT}
In this section we give sufficient conditions in order to have a central limit theorem
$$\frac{X_N-a_{N,\beta}}{\sqrt{v_{N,\beta}}} \rightharpoonup_{\substack{\proba_{(\beta,\delta_N)}\\N \to \infty}}\mathcal{N}(0,1),$$
and we compute an upper bound for the Kolmogorov distance between these two random variables thanks to Lemma \ref{lem:control_zone}.

\subsection{Control of the Fourier transforms}
Suppose first that $\beta=2$. By the Taylor integral formula,
$$\log \esper_{(2,\delta_N)}[\E^{\I\xi X_N}] = a_N\,\I \xi - \frac{v_N}{2}\,\xi^2 + \int_0^1\frac{(1-u)^2}{2}\,\Lambda_N'''(h_N+\I u \xi)\,(\I \xi)^3 \DD{u}.$$
Notice that for $\Re(z)>0$ the leading term of the asymptotic expansion of $\Lambda_N'''(z)$ provided before Proposition \ref{prop:mean_variance_beta2} is smaller in module than $\frac{1}{2|z|}$. Therefore in the Taylor integral formula the integral can be controlled as follows:
\begin{align*}
I &\leq \int_{0}^1 \frac{(1-u)^2}{4|h_N+\I u \xi|}\,|\xi|^3\DD{u} + O\!\left(\int_0^1 \frac{(1-u)^2}{2} \left|\frac{\xi}{h_N+\I u\xi}\right|^3 \DD{u}\right) \\
&\leq \int_{0}^1 \frac{1}{4\sqrt{(h_N)^2+u^2\xi^2}}\,|\xi|^3\DD{u} + O\!\left(\int_0^1 \frac{|\xi|^3}{((h_N)^2+u^2\xi^2)^{\frac{3}{2}}} \DD{u}\right) \\
&\leq \frac{\xi^2}{4} \int_0^{\frac{|\xi|}{h_N}} \frac{1}{\sqrt{1+v^2}}\DD{v} + O\!\left(\frac{\xi^2}{(h_N)^2}\int_0^{\frac{|\xi|}{h_N}}\frac{1}{(1+v^2)^{\frac{3}{2}}}\DD{v}\right) .\end{align*}
Thus, 
\begin{align*}
I &\leq \frac{\xi^2}{4}\left(\mathrm{arcsinh}\left(\frac{|\xi|}{h_N}\right)+O\!\left(\frac{|\xi|}{(h_N)^2\sqrt{(h_N)^2+\xi^2}}\right)\right) \\ 
&\leq \frac{\xi^2}{4}\left(\log\!\left(1+\frac{|\xi|}{\delta_N}\right)+O\!\left(\frac{|\xi|}{(\delta_N)^2\sqrt{(\delta_N)^2+\xi^2}}\right)\right),
\end{align*}
and we have proved:

\begin{proposition}[Control of the Fourier transform for Hua--Pickrell distributions]\label{prop:control_huapickrell}
Consider the random variable $X_N$ under a Hua--Pickrell distribution $\proba_{(2,\delta_N)}$. Uniformly for $\delta_N>0$, we have
$$\esper_{(2,\delta_N)}\!\left[\E^{\I \xi (X_N - a_N)}\right] = \exp\left(-\frac{v_N\,\xi^2}{2} + \xi^2\, O\!\left(\log\!\left(1+\frac{|\xi|}{\delta_N}\right)\right)\right),$$
and the $O(\cdot)$ in the equation above is actually smaller than 
$$\frac{1}{4} \left(\log\!\left(1+\frac{|\xi|}{\delta_N}\right)+\frac{C\,|\xi|}{(\delta_N)^2\sqrt{(\delta_N)^2+\xi^2}}\right)$$
for some constant $C>0$.
\end{proposition}
\medskip

In order to obtain an analogue proposition with $\beta\neq 2$, we use our Comparison Theorem \ref{thm:comparison}. By taking the third derivatives of the terms of the identity of this theorem, and by using the asymptotics of the polygamma functions, we obtain:
\begin{align*}
\Lambda_{N,\beta}'''(z) &= \frac{1}{(\beta')^2}\,\Lambda_N'''\!\left(\frac{z}{\beta'}\right) +  O_\beta\!\left(\frac{1}{|z|^2}+\frac{1}{(\Re(z))^3}\right).
\end{align*}
The Taylor integral formula for $\Lambda_{N,\beta}(h_N+\I \xi)$ gives then
$$\log \esper_{(\beta,\delta_N)}[\E^{\I\xi X_N}] = a_{N,\beta}\,\I \xi - \frac{v_{N,\beta}}{2}\,\xi^2 + \int_0^1\frac{(1-u)^2}{2}\,\Lambda_{N,\beta}'''(h_N+\I u \xi)\,(\I \xi)^3 \DD{u},$$
with an integral which is controlled by:
$$I \leq \frac{\xi^2}{4\beta'} \left(\log\!\left(1+\frac{|\xi|}{\delta_N}\right) + C_\beta \left(\frac{1}{\delta_N} \arctan\!\left(\frac{|\xi|}{\delta_N}\right) + \frac{|\xi|}{(\delta_N)^3}  \right)\right)$$
for some constant $C_\beta>0$. So:
\begin{proposition}[Control of the Fourier transform for circular Jacobi ensembles]\label{prop:control_circular_jacobi}
Consider the random variable $X_N$ chosen according to distribution of eigenvalues of the circular $(\beta,\delta_N)$ Jacobi ensemble. Uniformly for $\delta_N>0$ we have:
$$\esper_{(\beta,\delta_N)}\!\left[\E^{\I \xi (X_N-a_{N,\beta})}\right] = \exp\left(-\frac{v_{N,\beta}\,\xi^2}{2}+\xi^2\,O_\beta\!\left(\log\!\left(1+\frac{|\xi|}{\delta_N}\right)+\frac{\arctan(\frac{|\xi|}{\delta_N})}{\delta_N} +\frac{|\xi|}{(\delta_N)^3}\right)\right).$$
\end{proposition}

\subsection{Speed of convergence estimates}
Suppose that $\delta_N$ is bounded from below by a constant, say $\delta_N \geq 1$. Note then that in Propositions \ref{prop:control_huapickrell} and \ref{prop:control_circular_jacobi}, the remainder is always a $O_\beta(\frac{|\xi|}{\delta_N})$. Therefore Lemma \ref{lem:control_zone} ensures that under the law $\proba_{(\beta,\delta_N)}$,
\begin{equation}
\dkol\left(\frac{X_N-a_{N,\beta}}{\sqrt{v_{N,\beta}}}\,,\,\mathcal{N}_\R(0,1)\right)= O_\beta\!\left(\frac{1}{\delta_N\,(v_{N,\beta})^{\frac{3}{2}}}\right).\label{eq:kolmogorov}
\end{equation}

\begin{corollary}[Central limit theorem for large parameters $\delta_N$]\label{cor:central_limit}
Fix $\beta>0$ and consider a sequence of parameters $(\delta_N)_{N \in \N}$ such that $1\lesssim \delta_N  \ll N^{\frac{3}{2}}$. Under the laws $\proba_{(\beta,\delta_N)}$, the random variables $X_N$ are asymptotically normal:
$$V_N=\frac{X_N - \esper_{(\beta,\delta_N)}[X_N]}{\sqrt{\var_{(\beta,\delta_N)}(X_N)}} \rightharpoonup_{N \to \infty} \mathcal{N}_\R(0,1).$$
\end{corollary}

\begin{proof}
Suppose first that $\delta_N \ll N$. Then $(a_{N,\beta})_{N \in \N}$ is a balanced sequence and both terms of the product $\delta_N\,(v_{N,\beta})^{\frac{3}{2}}$ go to infinity; therefore, the Kolmogorov distance trivially goes to $0$ in this situation. If $\delta_N$ is of order $N$ but not larger, then we are in the regime of large deviations described by Proposition \ref{prop:large_regime}, and $v_{N,\beta}$ stays bounded from below while $\delta_N$ still goes to infinity: so, again, the Kolmogorov distance goes to $0$. We can finally focus on the case where $\delta_N\gg N$. If we rework the equation of Proposition \ref{prop:mean_variance_jacobi}, then we obtain the following estimate of the variance, which is slightly more precise than the one from Proposition \ref{prop:large_regime}:
$$v_{N,\beta} = \frac{1}{\beta}\,\log \left(1+\frac{1}{4\,\frac{\lambda_N}{N}(1+\frac{\lambda_N}{N})}\right) + O_\beta\!\left(\frac{N}{(\delta_N)^2}\right).$$
Taking the Taylor expansion of the logarithm yields
\begin{align*}
v_{N,\beta} &= \frac{N^2}{4\beta(\lambda_N)^2}\left(1+O_\beta\!\left(\frac{1}{N}+\frac{N}{\delta_N}\right)\right);\\
\delta_N\,(v_{N,\beta})^{\frac{3}{2}}&=M_\beta\,\frac{N^3}{(\lambda_N)^2}\left(1+O_\beta\!\left(\frac{1}{N}+\frac{N}{\delta_N}\right)\right)
\end{align*}
for some positive constant $M_\beta$. Thus as long as $\delta_N \ll N^{\frac{3}{2}}$, the estimate of the Kolmogorov distance ensures the asymptotic normality.
\end{proof}

\begin{remark}
If we take for instance $\delta_N=N^{\frac{5}{4}}$, then the Kolmogorov distance is a $O(N^{-\frac{1}{2}})$, but this central limit theorem is a bit strange: the variance of the random variable of interest $X_N$ under $\proba_{(\beta,\delta_N)}$ is in this case also a $O(N^{-\frac{1}{2}})$, so it goes to zero. Thus we have very small variables but which are still well-approximated by Gaussian distributions with adequate variances. We shall see in Section \ref{sec:sharp_polynomial} that the small variances prevent us to give exact asymptotics of the probabilities of large deviations in the regime $a_N=N$; in this case we shall only obtain upper bounds.
\end{remark}

\begin{remark}
The Berry--Esseen estimate can be made a bit more explicit if $1\lesssim \delta_N \lesssim N$. Indeed if $1\lesssim \delta_N \ll N$, then we are in the regime of moderate deviations and Proposition \ref{prop:log_N_aN} shows that $\beta\,v_{N,\beta}$ is equivalent to $\log(\frac{N}{\delta_N})$. 
Therefore with $V_N$ as in Corollary \ref{cor:central_limit} we have
\begin{equation}
\dkol(V_N,\mathcal{N}_{\R}(0,1))= O\!\left(\frac{1}{\delta_N \,(\log (\frac{N}{\delta_N}))^{\frac{3}{2}}}\right).\label{eq:berry_esseen_vn}
\end{equation}
This estimate also holds if $\delta_N$ is of order $N$: indeed if $bN\leq \delta_N \leq cN$ for some positive constants $b$ and $c$, then 
$(\log(\frac{N}{\delta_N}))^{\frac{3}{2}} $ and $ v_{N,\beta}$
are both bounded from below and from above by positive constants, so again they are of the same order. So, Equation \eqref{eq:berry_esseen_vn} holds as soon as $\frac{1}{c}\leq \delta_N \leq cN$ for some constant $c>0$, with an implied constant in the $O(\cdot)$ which depends only on $\beta$ and $c$.
\end{remark}
\bigskip

\section{Proof of the sharp moderate and large deviation principles}\label{sec:sharp_polynomial}

This last section is devoted to the proofs of our main Theorems \ref{thm:super_moderate} and \ref{thm:super_large}. Until the end of this section, $\beta>0$ is a fixed parameter and $(x_N)_{N \in \N} = (a_{N,\beta})_{N \in \N}$ is a sequence of positive numbers, which is supposed balanced in most of Subsection \ref{subsec:precise_moderate_deviations} ($\log N \lesssim a_{N,\beta} \ll N$), and of order $N$ in Subsection \ref{subsec:precise_large_deviations}.

\subsection{Precise moderate deviations}\label{subsec:precise_moderate_deviations}
Suppose that the sequence $(a_{N,\beta})_{N \in \N}$ is balanced. Then the parameter $\eps_N$ of the second step of the general scheme presented in Section \ref{subsec:general_method} can be taken equal to $\frac{1}{h_N\,(v_{N,\beta})^{\frac{3}{2}}}$ and by Proposition \ref{prop:log_N_aN} $v_{N,\beta}$ goes to infinity, so
$$\eps_N \ll \frac{1}{h_N\,\sqrt{v_{N,\beta}}} \ll 1.$$
So we have strong asymptotic normality after tilting. By following the arguments of the validity of the general scheme, we get:
\begin{align*}
\proba_{\beta}[X_N \geq a_{N,\beta}] &= \frac{\esper_{\beta}[\E^{h_NX_N-h_Na_{N,\beta}}] }{h_N\sqrt{2\pi v_{N,\beta}}}\left(1+O_\beta\!\left(\frac{1}{v_{N,\beta}}\right)\right) \\ 
&=\frac{\esper_{\beta}[\E^{h_NX_N-h_Na_{N,\beta}}]}{a_{N,\beta}}\,\sqrt{\frac{1}{2\pi\beta}\log\!\left(\frac{N}{a_{N,\beta}}\right)} \left(1+O_\beta\!\left(\frac{1}{\log(\frac{N}{a_{N,\beta}})}\right)\right).
\end{align*}
Above we go from the first line to the second line by using the following estimates
\begin{align*}
a_{N,\beta} &= a_N + O_\beta(1) =  \lambda_N\,\left(\log\!\left(\frac{N}{\lambda_N}\right) + O_\beta(1)\right);\\
\beta \,v_{N,\beta}&= \log\!\left(\frac{N}{\lambda_N}\right) + O_\beta(1).
\end{align*}
Multiplying the first line by $\frac{\beta}{h_N}=\frac{1}{\lambda_N}$ shows that 
\begin{align*}
\frac{a_{N,\beta}}{h_N} &= v_{N,\beta} + O_\beta(1);\\
h_N\,v_{N,\beta} &= a_{N,\beta} + O_\beta(h_N) = a_{N,\beta} \left(1+O_\beta\!\left(\frac{1}{\log(\frac{N}{a_{N,\beta}})}\right)\right);\\
h_N\,\sqrt{v_{N,\beta}} &= a_{N,\beta} \sqrt{\frac{\beta}{\log(\frac{N}{a_{N,\beta}})}}\left(1+O_\beta\!\left(\frac{1}{\log(\frac{N}{a_{N,\beta}})}\right)\right),
\end{align*}
whence the equation for $\proba_\beta[X_N \geq a_{N,\beta}]$. It remains to get rid of the tilting parameter $h_N$ in the Laplace transform $\esper_\beta[\E^{h_NX_N-h_Na_{N,\beta}}]$.  Notice that the logarithm of this quantity is the opposite of the Legendre--Fenchel conjugate $(\Lambda_{N,\beta})^*(a_{N,\beta})$, where 
$$(\Lambda_{N,\beta})^*(a) = \sup_{h \in \R} (ah-\Lambda_{N,\beta}(h)).$$
Indeed the parameter $h_N$ which maximises the function $h \mapsto a_{N,\beta}h-\Lambda_{N,\beta}(h)$ is the solution of $a_{N,\beta} = \Lambda_{N,\beta}'(h_N)=\esper_{(\beta,\delta_N)}[X_N]$, so we recover Equation \eqref{eq:critical_parameter}. So the previous estimate rewrites as:
\begin{equation}
\proba_{\beta}[X_N \geq a_{N,\beta}] = \frac{\E^{-\Lambda_{N,\beta}^*(a_{N,\beta})}}{a_{N,\beta}}\sqrt{\frac{1}{2\pi\beta}\log\!\left(\frac{N}{a_{N,\beta}}\right)} \left(1+O_\beta\!\left(\frac{1}{\log(\frac{N}{a_{N,\beta}})}\right)\right).\label{eq:almost_there}
\end{equation}
for any balanced sequence $(a_{N,\beta})_{N \in \N}$. This is the first item in Theorem \ref{thm:super_moderate}. In the sequel of this subsection we distinguish between several subregimes in order to prove the other items. 

\begin{remark}
Equation \eqref{eq:almost_there} also holds if $\sqrt{\log N} \ll a_{N,\beta} \ll \log N$ (in the regime of moderate deviations, we have $\log N \lesssim a_{N,\beta}$). Indeed if $a_{N,\beta} \ll \log N$, then Lemma \ref{lem:regime_sequences} shows that $\lambda_N$ goes to $0$, and
\begin{align*}
a_N &= (N+2\lambda_N)\,\psi_0(N+2\lambda_N) - (N+\lambda_N)\, \psi_0(N+\lambda_N)+\lambda_N\,\psi_0(\lambda_N) - 2\lambda_N\,\psi_0(2\lambda_N) \\ 
&= \lambda_N\,(\log N + \gamma + 1) + O\!\left((\lambda_N)^2+\frac{1}{N}\right).
\end{align*}
Combining this estimate and the relations between $a_N$ and $a_{N,\beta}$, we see that $\frac{1}{\sqrt{\log N}} \ll \lambda_N \ll 1$. Therefore, $v_N$ and $v_{N,\beta}$ are of order $\log N$, and Proposition \ref{prop:control_circular_jacobi} leads to:
$$\esper_\beta[\E^{\I \xi (X_N-a_{N,\beta})}] = \exp\left(-\frac{v_{N,\beta}\,\xi^2}{2} +  O_\beta\!\left(\left(\frac{|\xi|}{\lambda_N}\right)^3\right)\right).$$
Lemma \ref{lem:control_zone} yields 
$$\dkol\!\left(\frac{X_N-a_{N,\beta}}{\sqrt{v_{N,\beta}}},\,\mathcal{N}_\R(0,1)\right) = O\!\left(\frac{1}{(h_N\,\sqrt{v_{N,\beta}})^3}\right),$$ and then we see that the proof of the general scheme of approximation works again, since $h_N\,\sqrt{v_{N,\beta}}$ goes to infinity.
\end{remark}

\subsubsection{Small moderate deviations: $\sqrt{\log N} \ll a_{N,\beta} \lesssim \log N$.} Let us explain how to recover Equations \eqref{eq:moderate_deviation_small_range_haar} and \eqref{eq:moderate_deviation_small_range_jacobi}. We fix a constant $C$ such that $a_{N,\beta} \leq C \log N$; in the remainder of this paragraph, our $O(\cdot)$'s are allowed to depend on $C$ and $\beta$. By the second item of Lemma \ref{lem:regime_sequences} the sequence $(\lambda_N)_{N \in \N}$ is then bounded from above. Let us then find an asymptotic expansion of $\lambda_N$ in terms of $a_{N,\beta}$. We expect $$\lambda_N = \frac{a_{N,\beta}}{\log N}\,(1+\eta_N),$$ with $\eta_N$ small and of order $\frac{1}{\log N}$. In the sequel we use freely the relation $\psi_0(1+z)=\psi_0(z)+\frac{1}{z}$ and the fact that $\psi_0$ is Lipschitz in any interval $[1,M]$. Notice first that
\begin{align*}
a_{N} &= (N+2\lambda_N)\,\psi_0(N+2\lambda_N) - (N+\lambda_N)\,\psi_0(N+\lambda_N) + \lambda_N \,\psi_0(\lambda_N) - 2\lambda_N\, \psi_0(2\lambda_N) \\
&= \lambda_N \left(\log N +  1 +  \psi_0(1+\lambda_N)  - 2\, \psi_0(1+2\lambda_N)\right) + O\!\left(\frac{1}{N}\right).
\end{align*}
Then using the computations from the proof of Proposition \ref{prop:mean_variance_jacobi} we get:
\begin{align*}
a_{N,\beta} &= \lambda_N \left(\log N +  1 + \psi_0(1+\lambda_N)  - 2\, \psi_0(1+2\lambda_N)\right) \\
&\quad+\psi_0(1+2\beta'\lambda_N)-\psi_0(1+\beta'\lambda_N)+\frac{\beta'+1}{2\beta'}\,(\psi_0(1+\lambda_N)-\psi_0(1+2\lambda_N))\\
&\quad + \frac{1-\beta'^2}{12\beta'}\,(G_{N,\beta})'(2\beta'\lambda_N)+ O\!\left(\frac{1}{N}\right).
\end{align*}
If we replace on the last line $(G_{N,\beta})'(z)$ by its limit
$$H_{\beta}(z) = \int_0^\infty \eta_\beta(s)\,(\E^{-\frac{sz}{2}}-\E^{-sz})\DD{s},$$
then our error is again a $O(\frac{1}{N})$, so this replacement is legit. We then replace $\lambda_N$ by $\frac{a_{N,\beta}}{\log N}(1+\eta_N)$ in the formula above. We obtain:
\begin{align*}
-\eta_N&=   \frac{1 + \psi_0(1+\lambda_N)  - 2\, \psi_0(1+2\lambda_N)}{\log N}+\frac{\beta'+1}{2\beta'a_{N,\beta}}\left(\psi_0\!\left(1+\frac{a_{N,\beta}}{\log N}\right)-\psi_0\!\left(1+\frac{2a_{N,\beta}}{\log N}\right)\right) \\ 
&\quad +\frac{1}{a_{N,\beta}}\left(\psi_0\!\left(1+\frac{2\beta'a_{N,\beta}}{\log N}\right)-\psi_0\!\left(1+\frac{\beta'a_{N,\beta}}{\log N}\right)\right)+ \frac{1-\beta'^2}{12\beta'a_{N,\beta}}\,H_\beta\!\left(\frac{2\beta'a_{N,\beta}}{\log N}\right)\\ 
&\quad + O\!\left(\frac{1}{N}+\frac{\eta_N}{\log N}\right).
\end{align*}
All the functions considered above are Lipschitz on their domain of analysis, so all the terms of this estimate are of order $\frac{1}{\log N}$, and the remainder of the asymptotic expansion above is a $O(\frac{1}{(\log N)^2})$.
This exact formula will not be important in the sequel, as we shall only use the fact that $\eta_N=O(\frac{1}{\log N})$ (with again a constant which depends on $C$ and $\beta$). By Proposition \ref{prop:laplace_estimate_beta},
$$\Lambda_{N,\beta}(h_N) = \log \Psi_\beta\!\left(\frac{2\beta' a_{N,\beta}}{\log N}(1+\eta_N)\right)  + \frac{\beta'(a_{N,\beta})^2}{\log N} (1+2\eta_N) + O\!\left(\frac{1}{\log N}\right) $$
and on the other hand
$$h_Na_{N,\beta} = \frac{\beta'(a_{N,\beta})^2}{\log N}(2+2\eta_N),$$
so in this regime,
$$-\Lambda_{N,\beta}^*(a_{N,\beta}) = \Lambda_{N,\beta}(h_N)-h_Na_{N,\beta} = \log \Psi_\beta\!\left(\frac{2\beta' a_{N,\beta}}{\log N}(1+\eta_N)\right) -\frac{\beta'(a_{N,\beta})^2}{\log N}+ O\!\left(\frac{1}{\log N}\right).$$
Finally it is clear from the definitions of $\Psi$ and $\Psi_\beta$ that $\log \Psi_\beta(\cdot)$ is Lipschitz on the domain that we consider, so we can remove the factor $1+\eta_N$ from its argument. We have therefore proved:
\begin{proposition}
In the regime $\sqrt{\log N}\ll a_N \lesssim \log N$, with an implied constant which depends on $\beta$ and on the upper bound on the ratio $\frac{a_N}{\log N}$, we have:
$$\E^{-\Lambda_{N,\beta}^{*}(a_{N,\beta})} = \E^{-\frac{\beta'(a_{N,\beta})^2}{\log N}}\,\Psi_\beta\!\left(\frac{2\beta'a_{N,\beta}}{\log N}\right)\,\left(1+O\!\left(\frac{1}{\log N}\right)\right).$$
\end{proposition}
\noindent This ends the proof of the second point of Theorem \ref{thm:super_moderate} and we can express $\Psi_\beta(\frac{\beta a_{N,\beta}}{\log N})$ in terms of $\Psi(\frac{2a_{N,\beta}}{\log N})$ by using the first part of Proposition \ref{prop:laplace_estimate_beta}.

\subsubsection{True moderate deviations: $\log N \ll a_{N,\beta} \ll N$.}  We now focus on the second subregime of moderate deviations, which is when $a_{N,\beta}$ is much larger than $\log N$ but much smaller than $N$. We start by the following remark on the previous case: since for balanced sequences we have 
$$\beta\,v_{N,\beta} \simeq \log \left(\frac{N}{a_{N,\beta}}\right)\quad \text{and}\quad h_N\simeq \frac{\beta\, a_{N,\beta}}{\log(\frac{N}{a_{N,\beta}})},$$ 
in the regime of small moderate deviations, the previous computations show that
$$\Lambda_{N,\beta}(h_N)-h_Na_{N,\beta} = \log\Psi_\beta(h_N) - \frac{v_{N,\beta}\,(h_N)^2}{2} + \text{small remainder} .$$
This leads one to try to compare in the general case $-\Lambda_{N,\beta}^*(a_{N,\beta})$ and $\log \Psi_\beta(h_N)-\frac{v_{N,\beta}\,(h_N)^2}{2}$.

\begin{lemma}\label{lem:estimate_legendre}
Fix $\beta>0$, and consider a balanced sequence $(a_{N,\beta})_{N \in \N}$. Recall that $\eps_{\lambda_N}=\frac{\lambda_N}{N} = \frac{h_N}{\beta N}$. Then
\begin{align*}
&-\Lambda_{N,\beta}^*(a_{N,\beta})-\log \Psi_\beta(h_N) +\frac{v_{N,\beta} (h_N)^2}{2} \\
&=\frac{(h_N)^2}{4\beta'}\, \log \!\left(\frac{2h_N}{\beta'}\right)-\frac{3(h_N)^2}{8\beta'} -\frac{N^2\beta'-N(\beta'-1)}{2} \log\!\left(1+\frac{(\eps_{\lambda_N})^2}{1+2\eps_{\lambda_N}}\right)\\
&\quad -N(\beta'-1)\left(\eps_{\lambda_N}\,\log 2+\frac{(\eps_{\lambda_N})^2}{2(1+\eps_{\lambda_N})(1+2\eps_{\lambda_N})} \right) +\frac{3\beta'-1-\beta'^2}{8\beta'} + r_N,
\end{align*}
where the remainder $r_N$ is a $O_\beta((\eps_{\lambda_N})^2 + \frac{1}{h_N})$.
\end{lemma}

\begin{proof}
By Propositions \ref{prop:laplace_estimate_beta} and \ref{prop:mean_variance_jacobi},
\begin{align*}
\Lambda_{N,\beta}(h_N)&\simeq \log \Psi_\beta(h_N) + \frac{(h_N)^2\log N}{4\beta'} -\frac{3(h_N)^2}{8\beta'} + N^2\beta' \left(\frac{(1+2\eps_{\lambda_N})^2}{2} \,\ell_2-(1+\eps_{\lambda_N})^2\,\ell_1\right)\\
&\quad +N(\beta'-1)\left((1+\eps_{\lambda_N})\,\ell_1 -\frac{(1+2\eps_{\lambda_N})}{2}\ell_2\right);\end{align*}
\begin{align*}
-h_Na_{N,\beta}&\simeq -\frac{(h_N)^2}{2\beta'} \log \!\left(\frac{N\beta'}{2h_N}\right) + N^2\beta'\!\left( (2\eps_{\lambda_N}+2(\eps_{\lambda_N})^2)\,\ell_1-(2\eps_{\lambda_N}+4(\eps_{\lambda_N})^2)\,\ell_2\right) \\
&\quad +  N(\beta'-1)\,\eps_{\lambda_N} \,(\ell_2 -\ell_1 -\log 2) +\frac{3\beta'-1-\beta'^2}{12\beta'}; \\
\frac{v_{N,\beta}\, (h_N)^2}{2}&\simeq \frac{(h_N)^2}{4\beta'}\log\!\left(\frac{N\beta'}{2h_N}\right) 
+ N^2 \beta' \left(2(\eps_{\lambda_N})^2\,\ell_2 - (\eps_{\lambda_N})^2\ell_1\right)+ \frac{3\beta'-1-\beta'^2}{24\beta'} \\
&\quad-N(\beta'-1)\left(\frac{(\eps_{\lambda_N})^2}{2(1+\eps_{\lambda_N})(1+2\eps_{\lambda_N})} \right),
\end{align*} 
where $\ell_2=\log(1+2\eps_{\lambda_N})$, $\ell_1=\log(1+\eps_{\lambda_N})$, and the symbol $\simeq$ means that the two terms of the identity differ by a remainder $r_N$. We conclude by taking the sum of these three equations.
\end{proof}
\medskip

If $a_{N,\beta}\gg \log N$, then $\lambda_N$ goes to infinity by Lemma \ref{lem:regime_sequences} and we can use the second part of Proposition \ref{prop:laplace_estimate_beta} in order to replace in the formula above the quantity $\log \Psi_\beta(h_N)$ by an asymptotic equivalent:
\begin{align}
-\Lambda_{N,\beta}^*(a_{N,\beta}) &= -\frac{v_{N,\beta}\, (h_N)^2}{2} + \frac{\beta'^2-3\beta'+1}{12\beta'}\,\log h_N - \frac{N^2\beta'}{2} \log\!\left(1+\frac{(\eps_{\lambda_N})^2}{1+2\eps_{\lambda_N}}\right) \nonumber \\
&\quad + \frac{N(\beta'-1)}{2}\left(\log\!\left(1+\frac{(\eps_{\lambda_N})^2}{1+2\eps_{\lambda_N}}\right) -\frac{(\eps_{\lambda_N})^2}{(1+\eps_{\lambda_N})(1+2\eps_{\lambda_N})}\right) \nonumber \\
& \quad +B_\beta + O_\beta\!\left((\eps_{\lambda_N})^2+\frac{1}{h_N}\right), \label{eq:legendre_true_moderate}
\end{align}
where $B_\beta$ is the constant equal to
$$B_\beta = \frac{1-\beta'^2}{12\beta'} A_\beta + \frac{3\beta'-1-\beta'^2}{8\beta'} +\frac{3-\beta'}{12} \log 2 - \frac{3+2\beta'}{12}\log \beta' + \frac{\beta'-1}{4}\log \pi + \beta'\zeta'(-1) .$$
Next we have to replace each occurrence of $h_N$ by an adequate function of $a_{N,\beta}$. We therefore need to reverse the estimate of the mean from Proposition \ref{prop:mean_variance_jacobi}. This operation involves the map $\theta_{N,\beta}$ introduced at the beginning of Subsection \ref{subsec:main_results}.

\begin{lemma}\label{lem:modification_theta}
If the parameter $\beta>0$ is fixed, then for $N$ large enough, $\theta_{N,\beta}$ is a continuous increasing bijection between $\R_+$ and $[\frac{\beta'-1}{2\beta'N} \log 2,\log 2)$.
\end{lemma}

\begin{proof}
Let us prove that
$$\theta_{N,\beta}'(x)=\log\!\left(1+\frac{1}{4x(1+x)}\right)-\frac{\beta'-1}{2\beta'N(1+x)(1+2x)}$$
is strictly postive for any $x\in \R_+$ if $N$ is large enough; this will imply the result since the limits of $\theta_{N,\beta}$ when $x$ goes to $0$ and $+\infty$ are respectively $\frac{\beta'-1}{2\beta'N} \log 2$ and $\log 2$. For $x$ smaller than $\frac{1}{2}$, the logarithm is larger than $\log(\frac{4}{3})$, to which is subtracted a quantity smaller than $|\frac{\beta'-1}{2\beta'N}|$, so $\theta_{N,\beta}'(x)$ stays positive if $N$ is large enough. On the other hand by concavity of the logarithm, for $x$ larger than $\frac{1}{2}$, there is a positive constant $c$ such that
$$\log\!\left(1+\frac{1}{4x(1+x)}\right)\geq \frac{c}{4x(1+x)}\geq \frac{c}{2(1+x)(1+2x)},$$
which is again larger than $|\frac{\beta'-1}{2\beta'N(1+x)(1+2x)}|$ for $N$ large enough (depending on $\beta$, but not on $x$). 
\end{proof}\medskip

Denote $(\theta_{N,\beta})^{-1}$ the inverse function of $\theta_{N,\beta}$. We have
$$\theta_{N,\beta}(x) = x\,|\log x| +O_\beta\!\left(\frac{1}{N}+x\right),$$
so if $\frac{1}{N}\ll x \ll 1$, then for the same reasons as for the map $\theta$, we have the asymptotic equivalent
$$(\theta_{N,\beta})^{-1}(x) = \frac{x}{|\log x|}\,(1+o(1)).$$

\begin{lemma}
In the regime $\log N \ll a_{N,\beta} \ll N$, we have
$$\lambda_N =  N\, (\theta_{N,\beta})^{-1}\!\left(\frac{a_{N,\beta}}{N}\right) + O_\beta\!\left(\frac{1}{a_{N,\beta}}\right),$$
\end{lemma}

\begin{proof} 
 Our usual estimates of the mean $a_{N,\beta}$ (Proposition \ref{prop:mean_variance_jacobi} and Equation \eqref{eq:mean_huapickrell}) can be rewritten as:
$$\frac{a_{N,\beta}}{N} = \theta_{N,\beta}\left(\frac{\lambda_N}{N}\right) + O_\beta\!\left(\frac{1}{N\lambda_N}\right).$$
Since $\frac{a_{N,\beta}}{N}$ is much larger than $|\frac{\beta'-1}{2\beta'N} \log 2|$, we can apply the inverse of $\theta_{N,\beta}$ to this identity. We get
\begin{align*}
\frac{\lambda_N}{N} &= (\theta_{N,\beta})^{-1}\!\left(\frac{a_{N,\beta}}{N}\right) + O_\beta\!\left(\frac{|(\theta_{N,\beta}^{-1})'(\frac{a_{N,\beta}}{N})|}{N\lambda_N}\right) \\ 
&= (\theta_{N,\beta})^{-1}\!\left(\frac{a_{N,\beta}}{N}\right) + O_\beta\!\left(\frac{1}{|\theta_{N,\beta}'(\frac{\lambda_N}{N})|\,N\lambda_N}\right).
\end{align*}
However, for $x$ small we have $(\theta_{N,\beta})'(x)=|\log x|+O(1)$, so
$$\lambda_N =  N\,(\theta_{N,\beta})^{-1}\!\left(\frac{a_{N,\beta}}{N}\right) + O_\beta\!\left(\frac{1}{\lambda_N \,\log (\frac{N}{\lambda_N})}\right).$$
Finally, in the regime of moderate deviations, $a_{N,\beta}\simeq N\,\theta(\frac{\lambda_N}{N}) \simeq \lambda_N\,\log(\frac{N}{\lambda_N})$, so the remainder is a $O_\beta(\frac{1}{a_{N,\beta}})$.
\end{proof}
\medskip

We can make the previous estimate more precise and give the term of order $\frac{1}{a_{N,\beta}}$ in the asymptotic expansion:
\begin{proposition}
In the regime $\log N \ll a_{N,\beta} \ll N$, we have
$$\lambda_N =  N (\theta_{N,\beta})^{-1}\!\left(\frac{a_{N,\beta}}{N}\right) + \frac{3\beta'-1-\beta'^2}{24\beta'^2\,a_{N,\beta}}+ o\!\left(\frac{1}{a_{N,\beta}}\right).$$
\end{proposition}

\begin{proof}
A more precise version of the estimate of $a_{N,\beta}$ is:
$$\theta_{N,\beta}\!\left(\frac{\lambda_N}{N}\right) = \frac{a_{N,\beta}}{N} + \frac{3\beta'-1-\beta'^2}{24\beta'^2 N\lambda_N} + O\!\left(\frac{1}{N(\lambda_N)^2} + \frac{\lambda_N}{N^3}\right).$$
If we invert this relation we get:
\begin{align*}
\frac{\lambda_N}{N} &= (\theta_{N,\beta})^{-1}\!\left(\frac{a_{N,\beta}}{N}\right) + \frac{1}{\theta_{N,\beta}'((\theta_{N,\beta})^{-1}(\frac{a_{N,\beta}}{N}))}\left(\frac{3\beta'-1-\beta'^2}{24\beta'^2 N\lambda_N} + O\!\left(\frac{1}{N(\lambda_N)^2} + \frac{\lambda_N}{N^3}\right)\right)\\
&\quad +O\!\left(\frac{|((\theta_{N,\beta})^{-1})''(\frac{a_{N,\beta}}{N})|}{(N\lambda_N)^2}\right).
\end{align*}
For $x$ small, $\theta''(x) = -\frac{1}{x} + O(1)$, and the same estimate holds for $\theta_{N,\beta}''$. Therefore
$$((\theta_{N,\beta})^{-1})''\left(\frac{a_{N,\beta}}{N}\right) = -\frac{\theta_{N,\beta}''((\theta_{N,\beta})^{-1}(\frac{a_{N,\beta}}{N}))}{(\theta_{N,\beta}'((\theta_{N,\beta})^{-1}(\frac{a_{N,\beta}}{N})))^3} = O\!\left(\frac{1}{\frac{\lambda_N}{N}\,(\log (\frac{N}{\lambda_N}))^3}\right).$$
Thus, the remainder on the second line of the estimate of $\frac{\lambda_N}{N}$ is of order smaller than
 $$\frac{1}{N\,(\lambda_N\,\log(\frac{N}{\lambda_N}))^3} \lesssim \frac{1}{N\,(a_{N,\beta})^3}.$$
The other remainder is of order smaller than
$$\frac{1}{N(\lambda_N)^2\,\log (\frac{N}{\lambda_N})} + \frac{\lambda_N}{N^3 \log (\frac{N}{\lambda_N})} \lesssim \frac{\log (\frac{N}{\lambda_N})}{N\,(a_{N,\beta})^2} + \frac{(\frac{\lambda_N}{N})^2}{N\,a_{N,\beta}} .$$
In the regime $a_{N,\beta} \gg \log N$, $\frac{\log (\frac{N}{\lambda_N})}{(a_{N,\beta})^2} \leq \frac{\log N}{(a_{N,\beta})^2} \ll \frac{1}{a_{N,\beta}}$, so by gathering all the remainders we obtain:
$$\frac{\lambda_N}{N} = (\theta_{N,\beta})^{-1}\!\left(\frac{a_{N,\beta}}{N}\right) + \frac{1}{\theta_{N,\beta}'((\theta_{N,\beta})^{-1}(\frac{a_{N,\beta}}{N}))} \left(\frac{3\beta'-1-\beta'^2}{24\beta'^2 N\lambda_N}\right) + o\!\left(\frac{1}{N\, a_{N,\beta}}\right).$$
Finally we can replace up to a multiplicative $(1+o(1))$ 
$$\frac{1}{\theta_{N,\beta}'((\theta_{N,\beta})^{-1}(\frac{a_{N,\beta}}{N}))} \quad\text{by}\quad \frac{1}{\log (\frac{N}{\lambda_N})},$$
and then $\lambda_N\,\log(\frac{N}{\lambda_N})$ by its equivalent $a_{N,\beta}$ in order to obtain:
$$\frac{\lambda_N}{N} = (\theta_{N,\beta})^{-1}\!\left(\frac{a_{N,\beta}}{N}\right) + \frac{3\beta'-1-\beta'^2}{24\beta'^2\,N\,a_{N,\beta}} + o\!\left(\frac{1}{N\, a_{N,\beta}}\right).$$
Remultiplying by $N$ yields the desired asymptotic expansion.
\end{proof}\medskip

We can now demonstrate the third part of Theorem \ref{thm:super_moderate} by replacing in Equation \eqref{eq:legendre_true_moderate} all the occurrences of $h_N$ by the asymptotic expansion computed above. In the sequel we write $\vartheta_N = (\theta_{N,\beta})^{-1}(\frac{a_{N,\beta}}{N})$.

\begin{itemize}
    \item $ -\frac{v_{N,\beta}\, (h_N)^2}{2}$. Taking the square of the formula of the previous proposition we get:
    $$(h_N)^2 = 4(\beta')^2\left(N\vartheta_N\right)^2 +  \frac{(3\beta'-1-\beta'^2)\,N\vartheta_N}{3\,a_{N,\beta}} + o\!\left(\frac{1}{\log N}\right).$$
    On the other hand by using Proposition \ref{prop:mean_variance_jacobi} and Equation \eqref{eq:variance_huapickrell} we get:
    \begin{align*}
    v_{N,\beta} &= \frac{1}{2\beta'} \left(2\log(1+2\eps_{\lambda_N})-\log(1+\eps_{\lambda_N})-\log(4\eps_{\lambda_N})\right) +\frac{3\beta'-1-\beta'^2}{48\beta'^3(\lambda_N)^2} \\
    &\quad + \frac{\beta'-1}{2\beta'^2}\left(\frac{1}{2(N+\lambda_N)}  -\frac{1}{N+2\lambda_N} \right) + O_\beta\!\left(\frac{1}{N^2}+\frac{1}{(\lambda_N)^3}\right)\\
    &=\frac{1}{2\beta'} \left(2\log(1+2\vartheta_N)-\log(1+\vartheta_N)-\log(4\vartheta_N)\right) +\frac{3\beta'-1-\beta'^2}{48\beta'^3(\lambda_N)^2} \\
    &\quad + \frac{\beta'-1}{2\beta'^2}\left(\frac{1}{2(N+\lambda_N)}  -\frac{1}{N+2\lambda_N} \right) + O_\beta\!\left(\frac{1}{N^2}+\frac{1}{(\lambda_N)^2\, \log N}\right).
    \end{align*}
    When we multiply by $-\frac{(h_N)^2}{2}$, we obtain:
    \begin{align*}
    -\frac{v_{N,\beta}\,(h_N)^2}{2}&=  -\beta' (N \vartheta_N)^2 \log\left(1+\frac{1}{4\vartheta_N(1+\vartheta_N)}\right)+\frac{1-3\beta'+\beta'^2}{8\beta'}\\
    &\quad + \frac{N(\beta'-1)}{2}\left(\frac{(\eps_{\lambda_N})^2}{(1+\eps_{\lambda_N})(1+2\eps_{\lambda_N})} \right) + o(1).
    \end{align*}
    
    \item $\log h_N$. It is equal to $\log\beta + \log a_{N,\beta} - \log \log (\frac{N}{a_{N,\beta}}) +o(1)$.
    \item $\log(1+\frac{(\eps_{\lambda_N})^2}{1+2\eps_{\lambda_N}})$. We can replace it by $\log(1+\frac{(\vartheta_N)^2}{1+2\vartheta_N})$, because the difference is a $O(\frac{1}{N^2\,\log N})$.
\end{itemize}

\noindent So, $-\Lambda_{N,\beta}^*(a_{N,\beta})$ is equal to
\begin{align*}
 & -\beta' (N \vartheta_N)^2 \log\left(1+\frac{1}{4\vartheta_N(1+\vartheta_N)}\right) + \frac{\beta'^2-3\beta'+1}{12\beta'} \left(\log a_{N,\beta} - \log \log\!\left(\frac{N}{a_{N,\beta}}\right)\right) \\
&\quad- \frac{N^2\beta'-N(\beta'-1)}{2} \log\!\left(1+\frac{(\vartheta_N)^2}{1+2\vartheta_N}\right) \\
&\quad+  \frac{1-\beta'^2}{12\beta'} A_\beta  + \frac{\beta'-1}{4}\log \pi  + \frac{1}{12\beta'}\log 2 + \frac{1-6\beta'-\beta'^2}{12\beta'} \log \beta' + \beta'\zeta'(-1) + o(1).
\end{align*}
Replacing $\E^{-\Lambda_{N,\beta}^*(a_{N,\beta})}$ by this estimate in Equation \eqref{eq:almost_there}, we obtain the third item of Theorem \ref{thm:super_moderate}, with a constant $C_\beta$ equal to:
\begin{equation*}
C_\beta =  \frac{2^{\frac{1}{12\beta'}}\,\pi^{\frac{\beta'-3}{4}}}{\beta}\, \exp\left(\frac{1-\beta'^2}{12\beta'}(A_\beta+\log\beta') +\beta'\zeta'(-1)\right).
\end{equation*}\medskip

\subsubsection{Not too large moderate deviations: $\log N \ll a_{N,\beta} \lesssim N^{1/3}$.}
The probability computed in the third item of Theorem \ref{thm:super_moderate} is a function of $a_{N,\beta}$ which is explicit but a bit complicated. If $a_{N,\beta}$ is not too large, then we can simplify a lot this expression; this is the last part of Theorem \ref{thm:super_moderate}. Thus let us suppose that $\log N \ll a_{N,\beta} \lesssim N^{1/3}$. Then 
\begin{align*}
\log\!\left(1+\frac{1}{4\vartheta_N(1+\vartheta_N)}\right) &= \log\!\left(\frac{1}{4\vartheta_N}\right) + O(\vartheta_N);\\
-(N\vartheta_N)^2 \log\!\left(1+\frac{1}{4\vartheta_N(1+\vartheta_N)}\right) &= (N\vartheta_N)^2 \log(4\vartheta_N) + O(N^2(\vartheta_N)^3),
\end{align*}
and the remainder is a $O(\frac{(a_{N,\beta})^3}{N(\log N)^3})=o(1)$. Let us now inject $\vartheta_N$ in the equation that defines it. We use the fact that $\theta_{N,\beta}-\theta$ is a uniform $O(\frac{1}{N})$:
\begin{align*}
\vartheta_N \log(4\vartheta_N) &= (1+2\vartheta_N)\log(1+2\vartheta_N) - (1+\vartheta_N)\log(1+\vartheta_N) - \frac{a_{N,\beta}}{N} + O\!\left(\frac{1}{N}\right);\\
(N\vartheta_N)^2 \log(4\vartheta_N) &= (N\vartheta_N)^2 - a_{N,\beta}(N\vartheta_N) + O(N^2(\vartheta_N)^3+N(\vartheta_N)^2) \\ 
&=(N\vartheta_N)^2 - a_{N,\beta}(N\vartheta_N)+o(1).
\end{align*}
On the other hand,
\begin{align*}
- \frac{N^2\beta'-N(\beta'-1)}{2} \log\! \left(1+\frac{(\vartheta_N)^2}{1+2\vartheta_N}\right) &= -\frac{\beta'(N\vartheta_N)^2}{2}+O(N^2(\vartheta_N)^3+N(\vartheta_N)^2) \\
&= -\frac{\beta'(N\vartheta_N)^2}{2} + o(1).
\end{align*}
So the term in the exponential in the asymptotic expansion of $\proba_{\beta}[X_N \geq a_{N,\beta}]$ reduces to
$$\beta' \left(\frac{(N\vartheta_N)^2}{2} - a_{N,\beta}(N\vartheta_N)\right) + o(1).$$
This ends our study of the regime of moderate deviations.\medskip

\subsection{Precise large deviations}\label{subsec:precise_large_deviations}
In this last paragraph we suppose that $a_{N,\beta}=\alpha_0 N$ for some positive constant $\alpha_0>0$. If we go back to the tilting argument in the proof of validity of the general scheme, we see that everything works the same until we need to discard the remainder 
$$O_\beta(\eps_N) = O_\beta\!\left(\frac{1}{h_N\,(v_{N,\beta})^\frac{3}{2}}\right)$$ 
in the computation of the integral $I_N$ (the order of magnitude of $\eps_N$ is given by Equation \eqref{eq:kolmogorov}). When $a_{N,\beta}=O(N)$, the variance is a $O(1)$ and this remainder is not negligible anymore, but we can still compute an upper and lower bound. In the sequel, we focus on the upper bound; the proof of the lower bound relies on similar arguments and is sketched at the end of Paragraph \ref{subsubsec:tilting_variance_LD}. The reason why we need $\alpha_0$ small in order to get a lower bound is the following: unfortunately, our techniques are not sufficiently precise in order to always ensure that the remainder does not entirely compensate the main term in the estimation of $I_N$ (actually, the remainder does compensate the main term for $\alpha_0$ large, since $\proba_\beta[X_N \geq \alpha_0 N]=0$ when $\alpha_0 >\log 2$).\medskip

Henceforth we start from the inequality
$$\proba_{\beta}[X_N \geq \alpha_0 N] \leq \frac{\exp(-\Lambda_{N,\beta}^*(\alpha_0 N))}{\sqrt{2\pi v_{N,\beta}}\,h_N}\,\left(1+\frac{c_\beta}{v_{N,\beta}}\right),$$
where $c_\beta$ is some positive constant implied in the $O_\beta(\cdot)$ of the remainder in the estimate of $I_N$. 
Our goal is to obtain some explicit function $F(\alpha_0, N)$ such that the ratio
$$\frac{\proba_{\beta}[X_N\geq \alpha N]}{F(\alpha_0, N)}$$
is bounded from above. In Paragraph \ref{subsubsec:tilting_variance_LD} we compute in terms of $\alpha_0$ and $N$ an estimate of $h_N$ in the regime of large deviations, from which we derive an estimate of the variance $v_{N,\beta}$.  We then estimate $\Lambda_{N,\beta}^*(\alpha_0 N)$ in Paragraph \ref{subsubsec:rate_function}, which enables us to complete the proof of Theorem \ref{thm:super_large}.

\begin{remark}
In the following we require $\alpha_0$ to be always larger than some fixed positive quantity $\alpha$, and all the implied constants in the $O(\cdot)$'s depend implicity on this lower bound $\alpha$. Again, to make this clear, we indicate this by an index $\alpha$. This hypothesis allows us for instance to replace a $O_\beta(\frac{1}{\alpha_0 N})$ or $O_\beta(\frac{1}{h_N})$ by a $O_{\alpha,\beta}(\frac{1}{N})$.
\end{remark}

\subsubsection{Tilting parameter and variance in the regime of large deviations}\label{subsubsec:tilting_variance_LD}
In the regime of large deviations let us compute an asymptotic expansion of $\lambda_N$ up to order $O(\frac{1}{N})$. We have
\begin{align}
\alpha_0 &= \theta\!\left(\frac{\lambda_N}{N}\right) + \frac{\beta'-1}{2\beta'N} \left(\log 2 + \log(1+\eps_{\lambda_N}) - \log(1+2\eps_{\lambda_N})\right) + O_{\alpha,\beta}\!\left(\frac{1}{N^2}\right).\label{eq:to_invert}
\end{align}
Notice that in the framework of large deviations, we can work with the function $\theta$ instead of its modification $\theta_{N,\beta}$, because all the terms of the formulas which we shall manipulate have a well-identified order of magnitude which is a power of $N$. This was not the case in the framework of moderate deviations, and this is why until now we needed to work with $\theta_{N,\beta}$.
\medskip

As $\alpha_0$ and $\frac{\lambda_N}{N}$ are of order $O(1)$, the function $\theta$, its inverse $\theta^{-1}$ and their derivatives are Lipschitz on their domain of interest, so we can recursively compute the terms of the asymptotic expansion of $\frac{\lambda_N}{N}$. 
 The first order expansion is $\frac{\lambda_N}{N}=\theta^{-1}(\alpha_0)+\frac{L_1^*}{N}$ for some $L_1^*=O(1)$, and by replacing $\frac{\lambda_N}{N}$ by this formula in Equation \eqref{eq:to_invert}, we obtain:
$$\alpha_0 = \alpha_0 + \frac{\theta'(L_0)\,L_1^*}{N} + \frac{\beta'-1}{2\beta'N} \left(\log 2 + \log(1+L_0) - \log(1+2L_0)\right) + O_{\alpha,\beta}\!\left(\frac{1}{N^2}\right),$$
where $L_0=\theta^{-1}(\alpha_0)$. Thus by identification
\begin{align*}
L_1^*&=\frac{\beta'-1}{2\beta'}\,\frac{\log(1+2L_0)-\log(1+L_0)-\log 2}{2\log(1+2L_0)-\log(1+L_0)-\log(4L_0)}+O_{\alpha,\beta}\!\left(\frac{1}{N}\right);\\
\frac{\lambda_N}{N} &= L_0 + \frac{\beta'-1}{2\beta'N}\,\frac{\log(1+2L_0)-\log(1+L_0)-\log 2}{2\log(1+2L_0)-\log(1+L_0)-\log(4L_0)} + O_{\alpha,\beta}\!\left(\frac{1}{N^2}\right).
\end{align*}
On the other hand, we have
\begin{align*}
v_{N,\beta} &= \frac{1}{2\beta'} \log\!\left(1+\frac{1}{4\frac{\lambda_N}{N}(1+\frac{\lambda_N}{N})}\right) +O_{\alpha,\beta}\!\left(\frac{1}{N}\right) \\ 
&= \frac{1}{2\beta'} \log\!\left(1+\frac{1}{4L_0(1+L_0)}\right) + O_{\alpha,\beta}\!\left(\frac{1}{N}\right).
\end{align*}
As a consequence \begin{align*}
\left(1+\frac{c_\beta}{v_{N,\beta}}\right) &= O_{\alpha,\beta}((L_0)^2);\\
\frac{1}{\sqrt{2\pi v_{N,\beta}}\,h_N}\,\left(1+\frac{c_\beta}{v_{N,\beta}}\right) &= O_{\alpha,\beta}\!\left(\frac{(L_0)^2}{N}\right).
\end{align*}
We have thus demonstrated the first part of Theorem \ref{thm:super_large}. The proof of the lower bound is similar, working this time with a factor
$$1 - \frac{c_\beta}{v_{N,\beta}}.$$
For $\alpha_0$ in an interval $[\alpha,\alpha']$ with $\alpha'$ small enough, $v_{N,\beta}$ stays larger than $2c_\beta$, so this factor is greater than $\frac{1}{2}$. On the other hand, the ratio  $$\frac{h_{N}\,\sqrt{v_{N,\beta}}}{N}$$
stays bounded from above and from below for $\frac{a_N}{N}=\alpha_0 \in [\alpha,\alpha']$; this ends this sketch of proof of the second part of Theorem \ref{thm:super_moderate}.
\medskip

\subsubsection{Computation of the rate function}\label{subsubsec:rate_function} We now focus on the exponential term in our estimate of the probability $\proba_{\beta}[X_N \geq \alpha_0 N]$, and we suppose first that $\beta=2$. In this case, we can easily compute the term of order $\frac{1}{N^2}$ in the expansion of $\frac{\lambda_N}{N}$. Indeed by Proposition \ref{prop:mean_variance_jacobi}, for $\alpha_0\in [\alpha,\log 2)$,
$$\alpha_0 = \theta\!\left(\frac{\lambda_N}{N}\right) + \frac{1}{12N^2}\left(\frac{1}{1+\frac{\lambda_N}{N}} - \frac{1}{2\frac{\lambda_N}{N}} - \frac{1}{1+2\frac{\lambda_N}{N}}\right)+O_\alpha\!\left(\frac{1}{N^3}\right),$$
so if we set $\frac{\lambda_N}{N}=L_0+\frac{L_2^*}{N^2}$ with $L_0=\theta^{-1}(\alpha_0)$, then
$$\alpha_0 = \alpha_0 + \frac{\theta'(L_0)\,L_2^*}{N^2} + \frac{1}{12N^2}\left(\frac{1}{1+L_0} - \frac{1}{2L_0} - \frac{1}{1+2L_0}\right) + O_{\alpha}\!\left(\frac{1}{N^3}\right),$$
which gives by identification:
$$L_2^* = \frac{1}{12 \log\!\left(1+\frac{1}{4L_0(1+L_0)}\right)}\,\left(\frac{1}{1+L_0} - \frac{1}{2L_0} - \frac{1}{1+2L_0}\right) +O_{\alpha}\!\left(\frac{1}{N}\right).$$
Let us inject this formula in $\Lambda_N^*(\alpha_0 N)$. Note that when $z=h_N$ is of order $N$, the estimate from Proposition \ref{prop:laplace_estimate_haar} is not really precise because of the remainder $O(\frac{|z|^2}{N^2})$. However, we can use again Equation \eqref{eq:stirling_barnes} and the exact formula for $\Lambda_N$ to obtain:
\begin{align*}
\Lambda_N(h_N) &= -\frac{N^2}{2} \left((1+2\eps_{\lambda_N})^2\log (1+2\eps_{\lambda_N}) - 2(1+\eps_{\lambda_N})^2\log(1+\eps_{\lambda_N}) -2(\eps_{\lambda_N})^2 \log (4\eps_{\lambda_N}) \right) \\
&\quad - \frac{\log N}{12} + \frac{\log 2}{12}  + \frac{1}{12}\,\log\!\left(\frac{1+2\eps_{\lambda_N}+(\eps_{\lambda_N})^2}{\eps_{\lambda_N}(1+2\eps_{\lambda_N})}\right) + \zeta'(-1) +O\!\left(\frac{1}{N}\right).
\end{align*}
Substracting $a_Nh_N=2\alpha_0\lambda_N N$ yields:
\begin{align*}
 -\Lambda_{N}^*(\alpha_0N) 
&= -N^2\, I(L_0) - \frac{\log N}{12} + O_\alpha(1),
\end{align*}
where
$$I(x) = -\frac{1}{2}\left((1-4x^2)\log (1+2x) - 2(1-x^2)\log(1+x) +2x^2 \log (4x) \right).$$
Here we used the fact that $\alpha_0$ and $L_0$ are bounded from below; this yields adequate upper bounds on some functions and their derivatives in the formul{\ae} above. Up to a modification of the constant $M_\alpha$, this implies the third part of Theorem \ref{thm:super_large}.
\bigskip

Finally let us estimate $\Lambda_{N,\beta}^*(\alpha_0N)$ when $\beta \neq 2$. In this case, since $\frac{\lambda_N}{N}=L_0+O(\frac{1}{N})$, it is not very difficult to see that the asymptotic expansion of the Legendre--Fenchel transform will involve terms of order $N^2$ and terms of order $N$ (and then terms of order $O(\log N)$). This exact asymptotic expansion is a bit complicated, so let us focus only on the leading term of order $N^2$; this will lead us to a simple (not sharp) large deviation principle. The Comparison Theorem \ref{thm:comparison} yields:
\begin{align*}
\Lambda_{N,\beta}(h_N)&=\beta'\,\Lambda_N(2\lambda_N) +O(N)+\frac{1-\beta'^2}{12\beta'} \int_0^{\infty} \frac{(1-\E^{-s\beta'\lambda_N })^2}{s}\,\eta_\beta(s)\DD{s},
\end{align*}
and the end of Proposition \ref{prop:laplace_estimate_beta} ensures that the integral is a $O(\log N)$. Therefore
\begin{align*}
-\Lambda_{N,\beta}^*(\alpha_0N) &= \Lambda_{N,\beta}(h_N) - \alpha_0N\,h_N\\
&= -N^2\beta'\, I\!\left(\frac{\lambda_N}{N}\right) + O(N) =  -N^2\beta' \,I(L_0) + O(N).
\end{align*}
This ends the proof of Theorem \ref{thm:super_large}.
\bigskip
\bigskip

\printbibliography

\end{document}